%% file: article_final.tex
\documentclass[onefignum,onetabnum]{siamart250211}

\input{ex_shared}

\ifpdf
\hypersetup{
  pdftitle={Strategic Inference in Stackelberg Games},
  pdfauthor={R. Hu, D. Ralston, X. Yang and H. Zhou}
}
\fi

\usepackage{enumitem}

\makeatletter% 
\@mparswitchfalse%
\makeatother%
\begin{document}

\maketitle

% REQUIRED
\begin{abstract}
We study a continuous-time stochastic Stackelberg game in which a leader seeks to accomplish a primary objective while inferring a hidden parameter of a follower. The follower solves an entropy-regularized {linear–quadratic} tracking problem and responds to the leader’s trajectory with a randomized policy. Anticipating this response, the leader designs informative controls to maximize the estimation efficiency for the follower’s latent intent, through maximum likelihood estimation. Unlike prior work on discrete-time or finite-candidate inverse learning, our framework enables continuous parameter inference without prior assumptions and endogenizes the information source through the follower’s strategic feedback. We derive semi-explicit solutions, prove well-posedness, and develop recurrent neural network algorithms to approximate the leader’s path-dependent control. Numerical experiments demonstrate how the leader balances task performance and information gain, highlighting the practical value of our approach for adversarial strategic inference.
\end{abstract}

% REQUIRED

% REQUIRED

\section{Introduction}
Inferring an agent's intentions from limited interaction is a long-standing problem spanning statistics, control theory, and machine learning~\cite{bergemann_information_2019}. A central topic in this area is \emph{decision-making to learn}, i.e., designing models, controls, or experiments so that the resulting data is maximally informative. This idea appears prominently in optimal experimental design and its Bayesian variants, which select actions or system parameters to optimize information-based criteria~\cite{pukelsheim_optimal_2006,rainforth_modern_2023}. Related principles also arise in dual control, where actions must balance the competing goals of regulating system performance and reducing uncertainty about latent parameters~\cite{feldbaum1960dual, mesbah2018stochastic}. These decision-driven inference problems have broad relevance, with applications in cybersecurity~\cite{chen2019adversarial, liuIncentivebasedmodeling2005, shinde2021cyber}, defense planning, and autonomous systems~\cite{kim2025deceptive, ward2023active}, where understanding adversarial or partially observable behavior is essential for robust strategic responses.

In this paper, we propose a continuous-time Stackelberg game between {two strategic agents with adversarial objectives}, a leader and a follower, as a framework for strategic inference. The interaction unfolds over a finite time horizon $[0,T]$, and each agent’s state evolves according to controlled stochastic differential equations (SDEs).
{The leader moves first and commits to a control policy, generating a trajectory $x^L$ on $[0,T]$. Conceptually, we then formulate the follower’s problem as starting at its own time $0$, at which point the entire path $x^L$ on $[0,T]$ is realized and available and used as deterministic input to its control problem on $[0,T]$.}  
The follower aims to track the leader’s trajectory up to an unknown dilation parameter \(M \in \mathbb{R}\), which represents its latent intent or planning preference. 
{Specifically, the follower chooses an optimal stochastic policy $\pi^{F, \ast}$ that drives
$$  \ud X_t^F = \Big(A_F X_t^F  + B_F\int_\R y\,\pi^F_t(y)\ud y\Big)\ud t + \sigma_F \ud W^F_t,$$
in order to minimize
\begin{equation}\label{Ftask}
    \mathbb{E} \Big[\int_0^T \Big(\frac{Q_F}{2}\big(X^F_t - Mx^L_t\big)^2 +\frac{R_F}{2} \int_\R y^2 \,\pi_t^F(y)\ud y 
    + \lambda_F\int_\R  \pi_t^F(y)\log\pi_t^F(y) \ud y\Big) \ud t\Big].
\end{equation}}The leader, anticipating the follower’s response model, designs its control to simultaneously accomplish a tracking task and gain information about \(M\). 
{Specifically, the leader chooses an optimal control $u^{L, \ast}$
that drives
$$
 \ud X^L_t = \Big(A_L X_t^L + B_L u^L_t \Big) \ud t + \sigma_L \ud W^L_t,$$
so as to minimize a cost functional consisting of two components:
(i) a tracking objective that keeps $X^L$ close to a prescribed reference trajectory $F(t)$, and
(ii) an inference objective that promotes accurate estimation of the hidden parameter $M$ from the follower’s response $X^F$. The resulting trajectory $x^L$ is then announced and subsequently tracked by the follower.}
This formulation captures the key idea of \textit{acting to learn} in an adversarial context, where the leader’s trajectory influences the informativeness of the follower’s response. Thus, the strategic structure of the game reflects the leader’s dual objective: to control its own state effectively while shaping the follower’s observable behavior to reveal useful information. 

The interaction proceeds in three steps:
\begin{itemize}[leftmargin=4em]
\item[Step 1:] \textbf{Leader's move.} The leader selects a control policy {$u^{L, \ast}$ for the combined tracking and inference objectives (i)–(ii)} and generates its state trajectory {$x^L$}.
\item[Step 2:] \textbf{Follower's response.} {Given $x^L$,} the follower reacts by selecting a randomized control {$\pi^{F, \ast}$} that optimizes its own entropy-regularized cost {\eqref{Ftask}}.
\item[Step 3:] \textbf{Leader's inference.} The leader observes the follower's resulting trajectory {$x^F$} and estimates the hidden parameter $M$.
\end{itemize}
On anticipating the follower’s optimal response, the leader deliberately selects informative actions, observes the resulting feedback, and extracts information from the follower’s trajectory to estimate the hidden intent.

{Due to the complexity of the information structure in our model, we summarize all related assumptions in Table~\ref{tab:info_structure}; their detailed interpretations are discussed at their first occurrence in the subsequent context. 
Each of the assumptions falls into one of the three categories: (i) strategic awareness, which specifies whether the follower knows the leader’s available information and intentions; (ii) model knowledge, which describes which model parameters are known to the agents and which remain hidden; and (iii) observability, which determines the state information on which each agent bases its decisions or constructs estimators.}

\begin{table}[htbp]
\centering
\caption{{Information structures assumed for the leader and the follower.}}
\label{tab:info_structure}
\begin{tabular}{l|cc}
\toprule
\makecell{Agents' \\ information} & \makecell{\textbf{Leader}\\ (Rational \& strategic)} & \makecell{\textbf{Follower}\\ (Rational \& less-informed)} \\
\midrule
\makecell{\textbf{Strategic}\\ \textbf{awareness}} & 
\makecell{Aware of follower's obj. \eqref{eq:follower_objective_func}
} &
\makecell{Aware of leader's presence\\Unaware of leader’s knowledge\\
nor leader's obj. \eqref{eq:leader_var_objective_func}--\eqref{eq:leader_fisher_objective_func}
} \\
\midrule
\makecell{\textbf{Model} \\ \textbf{knowledge}} & \makecell{SDE forms of \(X^F,X^L\)\\ Leader \& follower cost forms\\
\(A_F,B_F,x^F_0,Q_F,R_F\)\\
\(A_L,B_L,x^L_0,Q_L,R_L,Q_{L,T}\)\\
\(\sigma_L,\lambda_L,F\)} & 
\makecell{
SDE form of \(X^F\)\\ Follower cost form\\
\(A_F,B_F,x^F_0,Q_F,R_F\)\\
\(\lambda_F,M\)} \\
\midrule
\makecell{\textbf{Observa-}\\\textbf{bility}} & \(u^L_t(t,X^L_{[0,t]}),\ \widehat{M}(x^F_{[0,T]})\) & \(\pi^F_t(t,X^F_t,x^L_{[0,T]})\)\\
\bottomrule
\end{tabular}
\end{table}

{Such strategic inference problems arise in a variety of adversarial settings. For instance, a defender may wish to infer the intent or capability of a potentially hostile agent by observing its motion over time, while the agent may attempt to conceal its objective by maintaining a stochastic response. Similar issues appear in cyber–defense, where a system administrator designs probing patterns and monitors responses from a potentially compromised component, and in financial markets, where a large trader or regulator seeks to infer another participant’s trading motive from order–flow data. Our “observe–then–track” Stackelberg formulation should be viewed as a stylized benchmark for these situations: the leader designs informative trajectories to elicit behavior from an adversarial follower, and then performs inference on the resulting trajectories.}

{Notably, our formulation differs from classical dynamic Stackelberg games, where the leader’s committed strategy is publicly announced and both agents make decisions in a non-anticipative manner based on simultaneous evolution of the system. 
In adversarial settings with information asymmetry, it is more realistic to assume that the leader’s internal control policy and objectives are not directly observable: the follower only observes realized trajectories and may wish to remain covert about its tracking intent. 
This is captured by our sequential “observe–then–track” structure: the leader first commits to a control, generating a trajectory $x^L$ on $[0,T]$, and the follower subsequently solves an entropy-regularized tracking problem conditional on the full observed path $x^L$. 
Extensions to more standard information structures under simultaneous system evolution are discussed in Section~\ref{sec:conclusions}.}

\smallskip

\noindent\textbf{Related Literature.}
Stackelberg games have a rich theoretical background. Early studies focused on discrete-time stochastic dynamics and continuous-time deterministic dynamics~\cite{basar_dynamic_1998,jb_cruz_survey_1975, simaan_stackelberg_1973}. Continuous-time stochastic Stackelberg games were first proposed in~\cite{bagchi_stackelberg_1981} and formalized in~\cite{yong_leader-follower_2002}, which triggered a growing body of research. Subsequent work addressed both open-loop solutions~\cite{oksendal_stochastic_2013, yong_leader-follower_2002} and closed-loop solutions~\cite{bensoussan2015maximum,hernandez_closed-loop_2024, li2023linear}, including mean-field extensions~\cite{li2024closed} and jump-diffusion settings~\cite{moon_linear-quadratic_2021}, with most of these models assuming that the leader and follower control a shared state process. By contrast, we consider two agents operating on distinct controlled dynamics that unfold sequentially.

Optimal experimental design and dual control also offer comparisons to our model. Optimal experimental design aims to reduce uncertainty while satisfying task objectives, typically by optimizing an information criterion~\cite{bergemann_information_2019, silvey_optimal_2013}. Dual control, originating from~\cite{feldbaum1960dual}, addresses the joint challenge of performance optimization and information acquisition~\cite{wittenmark1995adaptive}. In both cases, the uncertainty arises from exogenous noise, and the inference task is embedded within a non-strategic environment~\cite{bock_parameter_2013, mesbah2018stochastic, wilson_trajectory_2014}. Our model differs in that the leader's control shapes the information-generating behavior of an opponent, making the inference task endogenous to the interaction.

Recent advances in inverse reinforcement learning (IRL)~\cite{adams_survey_2022,arora_survey_2020,ng_algorithms_2000} and active exploration~\cite{lopes_active_2009} extend these ideas to settings where agent preferences are inferred from observed behavior. In contrast to previous work focusing on discrete-time models or limited observability~\cite{chen2019adversarial,liuIncentivebasedmodeling2005,shinde2021cyber,ward2023active}, our model analyzes continuous-time stochastic dynamics, where the system evolution is fully known up to a latent parameter. Furthermore, our inference is performed retrospectively after observing the full trajectory, rather than incrementally during interaction.

\smallskip
\noindent\textbf{Main Contributions.}
This paper makes two main contributions on modeling and computation.  
On the modeling side, we introduce a continuous-time stochastic Stackelberg game framework for strategic inference of an unknown parameter in an adversarial environment. Unlike discrete-time inverse learning formulations (e.g., \cite{ward2023active}), which rely on a finite candidate set for the latent parameter, our model permits unconstrained continuous inference without requiring prior distributional assumptions, thereby offering a more realistic and flexible characterization of adversarial behavior. In contrast to classical experimental design for diffusion processes, which studies inference in a fixed environment~\cite{hooker2015control}, our setting endogenizes the information source through the follower’s optimal response, explicitly incorporating strategic feedback. Moreover, we obtain semi-explicit solutions and establish global well-posedness results, providing theoretical guarantees for the existence and structure of optimal strategies.

On the numerical side, we develop learning-based methods tailored to the path-dependent nature of the leader’s control problem. Specifically, we approximate optimal controls using recurrent neural networks and benchmark their performance against linear-quadratic solutions. Numerical experiments highlight the trade-offs between task performance and information acquisition, offering quantitative insights into how the leader’s revealing strategy interacts with the follower’s tracking objective. These results demonstrate that our framework captures the essence of adversarial strategic inference, balancing the leader’s primary control objective with effective identification of the follower’s hidden intention.

\smallskip
\noindent\textbf{Notations.}
We fix a finite time horizon \([0,T]\). For any stochastic process \(\{X_s\}_{s\in[0,t]}\) (resp. \(\{X_s\}_{s\in[0,T]}\)), we write \(X_{[0,t]}\) (resp. \(X\)) and denote its realized sample path by \(x_{[0,t]}\) (resp. \(x\)). Superscripts \(L\) (resp. \(F\)) indicate quantities associated with the leader (resp. follower). Let \(\mathcal{C}_t\) denote the space of real-valued continuous functions on \([0,t]\). For \(f \in \mathcal{C}_t\), \(\|f\|_1\) and \(\|f\|_\infty\) denote its \(L^1\) and \(L^\infty\) norms, respectively. For vectors (resp. matrices), \(\|\cdot\|\) denotes the \(\ell^2\)-norm (resp. matrix 2-norm). We write \(\mathbb{S}^{N\times N}\) for the space of symmetric \(N\times N\) matrices, with inequalities understood in the positive semi-definite sense. The operator \(\diag\{\cdot\}\) maps a vector to a diagonal matrix, and for \(N\in\mathbb{N}\), we denote \([N] := \{1,2,\dots,N\}\).
Denote by \(\mathcal{P}(\R)\) the collection of probability measures on \(\R\).

\smallskip

The rest of the paper is organized as follows. 
Section~\ref{subsec:followers_problem} begins with the formulation of the follower's entropy-regularized tracking problem. Section~\ref{sec:leader_and_MLE} shifts to the leader's perspective: Section~\ref{subsec:MLE} derives the maximum likelihood estimator (MLE) for the follower’s latent parameter, and Section~\ref{subsec:leaders_problem} formulates the leader’s primary control objective. These components are integrated in Section~\ref{subsec:cost_functionals} into a unified stochastic control problem for the leader. Section~\ref{subsec:repeated_experiments} extends the framework to a multi-period setting. Numerical methods and results are presented in Section~\ref{sec:numerics}, and concluding remarks appear in Section~\ref{sec:conclusions}.

\section{The Follower’s Tracking Problem}\label{subsec:followers_problem}

In this section, we model the tracking problem faced by the follower, given the realization of the leader's state trajectory.
Using tools from stochastic control, we derive a semi-explicit solution for the follower's optimal policy and establish its global well-posedness, serving as the foundation for subsequent discussions.

Let $(\Omega,\F, \mathbb P)$ be a probability space supporting a Brownian motion $W^F$ in \(\R\). 
{Define $\F_t^F := \sigma(X_s^L, 0 \leq s \leq T) \vee \sigma(X_s^F, 0 \leq s \leq t)$. 
We interpret $\{\F_t^F\}_{t \in [0,T]}$ as the information available for the follower's decision-making problem at time $t$. In particular,  $\sigma(X_s^L, 0 \leq s \leq T) \subseteq \F_0^F$.
This augmentation reflects the key assumption that, at time 0,  the follower has already observed the realized trajectory \(x^L\) of the leader's state process \(X^L\) (see Section~\ref{subsec:leaders_problem} for its dynamics).
Consequently, throughout Section~\ref{subsec:followers_problem}, we treat $x^L$ as a given deterministic continuous function in the follower's control problem.}

The follower's state process $X^F$ evolves according to a randomized policy $\pi^F$, subject to exogenous noise modeled by \(W^F\):
\begin{equation}\label{eq:follower_explore_dynamics}
    \ud X_t^F = \Big(A_F X_t^F  + B_F\int_\R y\,\pi^F_t(y)\ud y\Big)\ud t + \sigma_F \ud W^F_t,\quad X^F_0 = x^F_0,
\end{equation}
where $A_F,B_F,x^F_0, \sigma_F \in\R$ and $\sigma_F>0$. Given a realization \(x^L\) of the leader's state process \(X^L\), the follower minimizes its expected cost of the form:
\begin{multline}\label{eq:follower_objective_func}
    J^F(\pi^F; x^L) := \mathbb{E} \Big[\int_0^T \Big(\frac{Q_F}{2}\big(X^F_t - Mx^L_t\big)^2 +\frac{R_F}{2} \int_\R y^2 \,\pi_t^F(y)\ud y\\
    + \lambda_F\int_\R  \pi_t^F(y)\log\pi_t^F(y) \ud y\Big) \ud t\Big],
\end{multline}
where $Q_F, R_F, \lambda_F > 0$, and {$M\in \mathbb R$ is the (fixed but unknown to the leader) dilation parameter encoding the follower's latent intent.}

By \emph{a randomized control} $\pi^F$, we mean a \(\mathcal{P}(\R)\)-valued process {that is progressively measurable with respect to $\{\F_t^F\}_t$} and can
be identified as a random product measure on $[0,T]\times \R$, whose projection on $[0, T]$ coincides with the Lebesgue measure.
For notational convenience, we do not distinguish between a measure and its associated density function. Accordingly, \(\pi^F_t(y)\) denotes the density associated with measure \(\pi^F_t\) evaluated at \(y\in\R\), with dependencies on the sample point \(\omega\in\Omega\) omitted.

In this paper, we further restrict attention to ``Markovian'' randomized control:
\begin{equation}
    \pi^F_t := \pi^F(\cdot \mid t, X_t^F),\quad \text{where}\quad \pi^F:[0,T]\times\R\to\mathcal{P}(\R).
\end{equation}
For well-posedness of the control problem, we assume that, the density \(\pi^F(\cdot \mid t, x) \) has a finite second moment and finite entropy, and satisfies the growth condition: \(\sup_{(t,x)\in[0,T]\times\R} (\int_\R |y|\,\pi^F(y\mid t,x)\ud y)/(1+|x|)<\infty\). For simplicity, we focus on the one-dimensional case and leave multidimensional extensions for future work.

\begin{remark}[Model interpretation]
\label{rem:interp}
In our model, states and controls can be interpreted as position and velocity, whose evolution follows the linear state dynamics~\eqref{eq:follower_explore_dynamics} under physical laws subject to random noise. 
The parameter \(\sigma_F\) quantifies the intensity of the exogenous noise and is unobservable to both agents.

{The adversarial aspect enters through the information asymmetry between the two agents. The follower is assumed to be aware of the potential presence of the
adversarial leader and observe the realized trajectory \(x^L\) of the leader’s state process \(X^L\), while remaining unaware of the leader’s information structure and inference objective. In response to this uncertainty, the follower adopts a randomized decision rule \(\pi_t^F\). The role of this randomization is not exploratory, but strategic: by introducing stochasticity into its control, the follower seeks to obscure its underlying intent and reduce the leader’s ability to infer its strategy. Owing to the measurability constraint discussed in Remark~\ref{rem:measurability}, the state dynamics~\eqref{eq:follower_explore_dynamics} depend only on the mean of the randomized actions, rather than on individual action realizations.}

The expected cost $J^F$ encodes the follower's incentives: tracking a dilated version of the leader's trajectory $x^L$ while minimizing control efforts, subject to the maximum entropy principle (MEP) \cite{jaynes1957information}.
Throughout the game, the dilation factor \(M\), which encodes the follower's preference, {and the follower's risk-sensitivity parameter \(\lambda_F\), remain unknown to the leader.}

It is worth noting that, in the reinforcement learning (RL) literature, the MEP is commonly used to encourage exploration in unknown environments \cite{eysenbach2021maximum,hazan2019provably,wang2020reinforcement}. {Works such as \cite{szpruch2024optimal,giegrich2024convergence,guo2026fast} develop model-free learning schemes for entropy-regularized LQ control problems.} In contrast, our model assumes that both agents possess {certain} knowledge  of the environment defined by~\eqref{eq:follower_explore_dynamics}--\eqref{eq:follower_objective_func}, while the follower retains policy stochasticity to hide its intention rather than to explore. {Our framework could also be extended to a model-free setting in which the follower learns its optimal randomized tracking policy via RL; such an extension is briefly discussed in Section~\ref{sec:conclusions}.}
\end{remark}

\begin{remark}[Measurability issue]
    \label{rem:measurability}
    For a continuum of independently sampled actions \(u^F_t\sim \pi^F_t,\ \forall t\in[0,T]\), the mapping \(t\mapsto u^F_t\) is generally non-measurable with respect to the Lebesgue measure, except in trivial cases where the flow of measures \(\pi^F\) degenerates \cite{sun2006exact}.
    Consequently, replacing \(\int_\R y\,\pi^F_t(y)\ud y\) in the drift of~\eqref{eq:follower_explore_dynamics} by the random draw \(u^F_t\) destroys measurability,  and the existence of a measurable solution \(X^F\) is no longer guaranteed.

    Classical remedies for this issue include the Fubini extension and the exact law of large numbers \cite{sun2006exact,sun2009individual}, but come at the cost of extending the Lebesgue measure. We refer interested readers to \cite{carmona2025reconciling,frikha2023actor} for additional discussions and examples.

    To address this issue while preserving the adversarial interpretation of our model, we adopt the framework of relaxed controls \cite{fleming1984stochastic,nicole1987compactification,wang2020reinforcement,zhou1992existence}, in which the control enters the dynamics through the mean of the randomization rather than the random sample itself.
    In particular, when \(\pi^F(\cdot\mid t,x) = \mathcal{N}(m^F_t(x),(\Sigma^F)^2)\) is Gaussian, as in the optimal case established in Proposition~\ref{prop:follower_optimal_control}, the control can be written as
    \[u^F_t = m^F_t(X^F_t) + \Sigma^F\xi_{t}, \quad \xi_t\overset{\mathrm{i.i.d.}}{\sim}\mathcal{N}(0,1) \text{ independent of \(W^F\).} \]
    An Euler scheme of~\eqref{eq:follower_explore_dynamics} on a time grid {\(\Delta := \{(k-1)h:k\in[N_T]\}\) with step size $h:= T/N_T$} yields the accumulated randomization term \(\sum_{t\in\Delta}\xi_t h \sim \mathcal{N}(0,Th)\), which vanishes as \(h\to 0\).
    In other words, on a fine enough time discretization grid (as the continuous-time limit), using controls \(m^F_t(x)\) or \(u^F_t\) yields identical state trajectories.
\end{remark}

Based on dynamic programming, we derive Proposition~\ref{prop:follower_optimal_control}, providing a semi-explicit solution of the follower's optimal policy. The detailed proof is deferred to Appendix~\ref{sec:appendix_a}.

\begin{proposition}
\label{prop:follower_optimal_control}
    Given the leader's state trajectory $x^L$, the follower's optimal Markovian randomized policy \(\pi^{F,*}\) for problem~\eqref{eq:follower_explore_dynamics}--\eqref{eq:follower_objective_func} is given by
    \begin{equation}
        \label{eqn:F_opt_pi}
      \pi^{F,*}(\cdot\mid t, x) =
      \mathcal N\Big(-\frac{B_F}{R_F}\bigl(2a_t x +b_t\bigr),\frac{\lambda_F}{R_F}\Big),
    \end{equation}
    where $a_t$ and $ b_t$ satisfy the ordinary differential equations (ODEs):
    \begin{align}
        \label{eq:follower_ODE_sys_a}
      \dot{a}_t &= \frac{2B_F^{2}}{R_F}a^{2}_t - 2A_F a_t -\frac{Q_F}{2},\quad \dot{b}_t =  \frac{2B_F^{2}}{R_F}a_tb_t - A_Fb_t + Q_F M x_t^L,
    \end{align}
    with terminal conditions \(a_T = 0,\ b_T = 0\).
\end{proposition}

    Denote by \(X^{F,*}\) the follower's optimal state process induced by \(\pi^{F,*}\), which satisfies the linear SDE:
    \begin{equation}
        \label{eqn:opt_state}
        \ud X_t^{F,\ast} = \Big(f(t)X_t^{F,\ast} - \frac{B_F^2}{R_F} b_t\Big) \ud t + \sigma_F \ud W_t^F,\quad \text{where}\quad f(t) := A_F - \frac{2B_F^2}{R_F}a_t.
    \end{equation}
    This is an Ornstein-Uhlenbeck process with the closed-form solution
    \begin{equation*} 
        X^{F,*}_t = x^F_0e^{\int_0^t f(s)\ud s} - \frac{B_F^2}{R_F}\int_0^t e^{\int_s^t f(u)\ud u} b_s\ud s \\
        + \sigma_F \int_0^t e^{\int_s^t f(u)\ud u}\ud W^F_s.
    \end{equation*}
    Hence \(X^{F,*}_t\) is Gaussian with mean \( x^F_0e^{\int_0^t f(s)\ud s} - \frac{B_F^2}{R_F}\int_0^t e^{\int_s^t f(u)\ud u} b_s\ud s\) and variance \(\sigma_F^2 \int_0^t e^{2\int_s^t f(u)\ud u}\ud s\).
    {Importantly, \(X^{F,*}\) depends on \(M\) completely through \(b_t\), and is independent of the follower's risk-sensitivity parameter \(\lambda_F\); in particular, the leader has complete knowledge of the SDE form of the follower's optimal state process, regardless of the follower's level of risk sensitivity.}

The existence of the follower’s optimal policy \(\pi^{F,*}\) relies on the solvability of the ODE system~\eqref{eq:follower_ODE_sys_a}. We establish its global well-posedness in the following theorem, with the proof provided in Appendix~\ref{sec:appendix_a}.

\begin{theorem}
\label{thm:follower_ODE}
    The ODE system~\eqref{eq:follower_ODE_sys_a} admits a unique solution on $[0,T]$, $\forall T>0$.  
\end{theorem}

\section{The Leader's Inference and Optimization}\label{sec:leader_and_MLE}
This section presents the leader's perspective in the Stackelberg game, building on the follower’s tracking problem in Section~\ref{subsec:followers_problem}.
With full knowledge about the follower's problem and optimal solution (except for the unknown parameter $M$), the leader strategically designs its control at the outset of the game to extract the maximal information about $M$.

Specifically, Section~\ref{subsec:MLE} introduces the maximum likelihood estimator (MLE) for $M$, computed at the final step of the game based on the realization of the follower's state trajectory.
Section~\ref{subsec:leaders_problem} defines the leader's primary task, which is then integrated with two information metrics (the estimator variance and Fisher information) in Section~\ref{subsec:cost_functionals}, leading to the complete formulation of the leader's optimization problem. Finally, Section~\ref{subsec:repeated_experiments} extends the framework to a multi-period setting.

\subsection{Maximum Likelihood Estimator for \(M\)}\label{subsec:MLE}
At the final step of the game, the leader constructs an MLE for \(M\), after the sample path of \(X^{F, \ast}\) has been realized.

In the follower's optimal state dynamics~\eqref{eqn:opt_state}, \(b_t\) is the only component that admits dependencies on \(M\), which is unknown to the leader. Solving the ODE~\eqref{eq:follower_ODE_sys_a} for \(b_t\) gives 
\begin{equation}
    \label{eqn:g}
    b_t = M\Big(-Q_F\int_t^T e^{\int_t^s f(u) \ud u}x_s^L \ud s \Big)=:M\,g(t,x^L),
\end{equation} 
where \(g:[0,T]\times \mathcal{C}_T\to \R\) does not depend on parameter \(M\).

{Next, we compute the log-likelihood of observing \(X^{F,*}\) under \(M\), conditional on the leader's trajectory \(x^L\). Let \(\mu_{X^{F,*}}\) (resp. \(\mu_{\sigma_FW^F}\)) denote the law of \(X^{F,*}\) (resp. \(\sigma_FW^F\)) as a probability measure on the path space \((\mathcal{C}_T,\mathscr{B}_T)\), where \(\mathscr{B}_T\) is the associated Borel \(\sigma\)-field on \(\mathcal{C}_T\).
By a Girsanov-type argument \cite[Theorem~1.12]{kutoyants2013statistical} applied to equation~\eqref{eqn:opt_state}, the log-Radon-Nikodym derivative \(\ell(X^{F,*}|X^L = x^L;M) := \log \frac{\ud \mu_{X^{F,*}}}{\ud \mu_{\sigma_F W^F}}\) exists and admits the following representation:}
\begin{multline} \label{eq:log_likelihood}
  \ell(X^{F,*}|X^L = x^L;M) = \frac{1}{\sigma^2_F} \int_0^T \Big(f(t)X^{F,*}_t - \frac{B_F^2}{R_F}Mg(t,x^L)\Big) \ud X^{F,*}_t \\
    - \frac{1}{2\sigma^2_F} \int_0^T \Big(f(t)X^{F,*}_t - \frac{B_F^2}{R_F}Mg(t,x^L)\Big)^2 \ud t. 
\end{multline}
{We remark that: (i) although \cite[Theorem~1.12]{kutoyants2013statistical} is derived as a corollary of the Girsanov theorem, the Novikov condition need not be verified explicitly. 
Indeed, via a localization argument (see, e.g., \cite[Theorem~7.19]{liptser1977statistics}), the local boundedness and quadratic growth of the drift coefficient in equation~\eqref{eqn:opt_state} are sufficient.
(ii) The choice of the reference measure is, to some extent, arbitrary and does not affect the computation of the maximum likelihood estimator of \(M\). Therefore, we adopt \(\mu_{\sigma_FW^F}\) here as a natural reference measure.}

Given observations of both \(X^{F,*}\) and \(X^{L}\), the leader's MLE for \(M\) should maximize the joint log-likelihood {of \((X^{F,*},X^L)\), denoted by $\ell(X^{F,*},X^{L};M)$, which admits the standard conditional decomposition:}
\begin{equation}
    \ell(X^{F,*},X^{L};M) = \ell(X^{F,*}|X^{L};M) + \ell(X^{L};M),
\end{equation}
{where \(\ell(X^{L};M)\) denotes the marginal likelihood of observing \(X^L\), which remains} independent of  \(M\), due to the leader's inability of accessing \(M\) during its decision-making process (cf. Section~\ref{subsec:leaders_problem}).
Here, the MLE \[\widehat{M}:= \argsup_M \ell(X^{F,*},X^{L};M)\] is obtained by  maximizing \eqref{eq:log_likelihood}, which is quadratic in $M$, and reads
\begin{equation}\label{eq:MLE_estimator}
  \widehat{M} = \frac{\int_0^T f(t)\,g(t, X^L)X_t^{F, *} \ud t - \int_0^T g(t, X^L) \ud X_t^{F, *}}{\frac{B_F^2}{R_F}\int_0^T g^2(t, X^L) \ud t}. 
\end{equation}
This estimator is almost surely well-defined since \(g\not\equiv 0\ \mathrm{a.s.}\).

Proposition~\ref{prop:MLE_calc} below summarizes the mean, variance, and Fisher information of this MLE, with proof given in Appendix~\ref{sec:appendix_b}.
\begin{proposition}
    \label{prop:MLE_calc}
    The mean, variance, and Fisher information of the MLE~\eqref{eq:MLE_estimator} are
    \begin{align}
        \label{eqn:MLE_var}
        &\E[\widehat{M}] = M,\qquad  \mathrm{Var}[\widehat{M}] = \frac{\sigma_F^2R_F^2}{B_F^4}\,\E\Bigg[\frac{1}{\int_0^T g^2(t, X^L) \ud t}\Bigg],\\
        \label{eqn:MLE_FI}
        &I(M) = \frac{B_F^4}{\sigma_F^2R_F^2} \E\Bigg[\int_0^T g^2(t,X^L)\ud t\Bigg],
    \end{align}
    where the Fisher information is \(I(M) := \E \Big[(\partial_M \ell(X^{F,*},X^L;M))^2\Big]\).
\end{proposition}

\begin{remark}[Observability of \(\sigma_F\)]\label{rem:discrete_time}
Although \(\sigma_F\) representing the magnitude of exogenous noises (cf. Remark~\ref{rem:interp}) remains initially unknown to both agents, the leader can perfectly estimate \(\sigma_F\) using the quadratic variation, as long as it has continuous observations of the follower's realized state trajectory \(x^{F,*}\) \cite[Theorem~4]{florens1989approximate}. 

    Therefore, all model parameters appearing in the variance and Fisher information, including \(B_F,R_F\), already assumed known to the leader, are accessible. This justifies our ability to integrate them into the leader's optimization problem.

    In Appendix~\ref{sec:appendix_c}, we further consider a realistic extension in which \(x^{F,*}_t\) is observed only at discrete times. In this setting, the quadratic variation no longer yields an exact estimate of \(\sigma_F\).
    We derive joint MLEs for \(M\) and \(\sigma_F\), and show that as the observation grid becomes denser the MLE for \(M\) (resp. for \(\sigma_F\)) converges to \eqref{eq:MLE_estimator} (resp. to the true \(\sigma_F\)).
    This establishes the continuous-observation setting in our model as the limiting case of the discrete-observation model.
\end{remark}

\subsection{The Leader's Primary Task}\label{subsec:leaders_problem}
Consider the same probability space \linebreak \((\Omega, \F, \Prob)\), which supports another one-dimensional Brownian motion $W^L$ independent of $W^F$. {Here, $\{\F_t^L\}$ denotes the filtration generated by \(X^L\). 
The leader is assumed to know about the follower's control objective~\eqref{eq:follower_explore_dynamics}--\eqref{eq:follower_objective_func} and all the model parameters within, except for the latent parameter $M$ and the risk-sensitivity parameter \(\lambda_F\).}
The leader's state $X^L$ evolves according to the control process $u^L$:
\begin{equation}\label{eq:leader_dynamics}
    \ud X^L_t = \Big(A_L X_t^L + B_L u^L_t \Big) \ud t + \sigma_L \ud W^L_t,\quad X^L_0 = x^L_0,
\end{equation}
where $A_L, B_L,x^L_0, \sigma_L \in \R$ and $\sigma_L > 0$. 

The leader's primary objective is to track a continuous deterministic trajectory \(F:[0,T] \to \R\) while minimizing control efforts. This is represented by:
\begin{equation}
    \label{eqn:leader_primary}
    J^L_{\mathrm{P}}(u^L) := \E\Big[ \int_0^T \Big(\frac{Q_L}{2}(X^L_t - F(t))^2 + \frac{R_L}{2}(u^L_t)^2\Big)\ud t + \frac{Q_{L,T}}{2}(X^L_T - F(T))^2\Big],
\end{equation}
where \(Q_L,R_L,Q_{L,T}>0\). The interpretation of the leader's dynamics and costs parallels that of the follower (in Remark~\ref{rem:interp}), and is therefore omitted here.

Motivated by the path-dependent feature of the information metrics~\eqref{eqn:MLE_var}--\eqref{eqn:MLE_FI} on $X^L$ through $g$ (cf. \eqref{eqn:g}), we consider closed-loop controls {$u_t^L \in \F_t^L$} of the form
\begin{equation}
    \label{eqn:leader_info}
    u^L_t := u^L(t,X^L_{[0,t]}),\quad \text{where}\quad u^L:[0,T]\times \mathcal{C}_t\to\R,
\end{equation}
and depends only on the full history of its own state. For well-posedness, the measurable function \(u^L\) is assumed to satisfy the growth condition \(\sup_{(t,\gamma)\in[0,T]\times\mathcal{C}_t} |u^L(t,\gamma)|\) \(/(1+ \sup_{s\in[0,t]}|\gamma(s)|)<\infty\).

\subsection{The Leader's Strategic Inference}\label{subsec:cost_functionals}
Acting at the beginning of the Stackelberg game, the leader pursues two objectives: fulfilling the primary task while maximizing efficiency of the final estimation step.
In other words, the leader takes into account the follower's potential reactions, and optimally designs its control to extract as much information as possible about the follower’s parameter \(M\).

Two information metrics, the variance and the Fisher information (FI), derived in Proposition~\ref{prop:MLE_calc}, form the basis of the leader’s inference objective.
A smaller variance corresponds to a more efficient estimator, while a larger FI implies higher asymptotic efficiency \cite{kutoyants2013statistical}.

\smallskip
\noindent\textbf{Variance Minimization.}
The leader's first formulation seeks to minimize the variance of the estimator:
\begin{equation}\label{eq:leader_var_objective_func}
  J_{\Var}^L(u^L) := \lambda_L\frac{\sigma_F^2R_F^2}{B_F^4}\,\E\Bigg[\frac{1}{\int_0^T g^2(t, X^L) \ud t}\Bigg] + J^L_{\mathrm{P}}(u^L),
\end{equation}
where \(\lambda_L\geq 0\) quantifies the intensity of strategic inference. Due to the reciprocal integral term, this optimization problem rarely admits a semi-explicit solution and will be addressed numerically in Section~\ref{sec:numerics}.

\smallskip
\noindent\textbf{Fisher Information Maximization.}
Alternatively, the leader may maximize the Fisher information, leading to the problem
\begin{equation}\label{eq:leader_fisher_objective_func}
  J_I^L(u^L) := -\lambda_L \frac{B_F^4}{\sigma_F^2R_F^2} \E\Bigg[\int_0^T g^2(t,X^L)\ud t\Bigg] + J^L_{\mathrm{P}}(u^L),
\end{equation}
where \(\lambda_L\geq 0\) again measures the strength of the inference incentive. By Jensen’s inequality, $$\E\big[\tfrac{1}{\int_0^T g^2(t, X^L) \ud t}\big] \geq \tfrac{1}{\E[\int_0^T g^2(t, X^L) \ud t]}.$$
Hence, maximizing FI~\eqref{eq:leader_fisher_objective_func} can be viewed as minimizing a lower bound of the variance~\eqref{eq:leader_var_objective_func}.

The advantage of considering FI maximization~\eqref{eq:leader_fisher_objective_func} lies in its potential analytical tractability. Despite the inherently path-dependent nature of the problem, we augment the state variables by introducing auxiliary path-dependent components to Markovianize the dynamics. This reformulation enables the construction of a semi-explicit Markovian optimal control in terms of the augmented states (cf. Proposition~\ref{prop:leader_optimal_control}). Such a construction serves two purposes.
First, it provides a semi-explicit optimal control for~\eqref{eq:leader_fisher_objective_func} within a subclass of admissible controls~\eqref{eqn:leader_info}, offering a potentially optimal candidate among all path-dependent decision rules.
Second, it offers a natural benchmark for the deep learning algorithm developed in Section~\ref{sec:numerics} to address the variance-based formulation~\eqref{eq:leader_var_objective_func}.

The rest of this section focuses on such derivations. To this end, we augment the leader’s state by introducing processes $Y$ and $Z$:
\begin{equation}
    \label{eqn:Y_Z}
    Y_t := -\int_0^t h(s) X_s^L \ud s, \quad Z_t := \int_0^t k(s) Y_s \ud s,\ \forall t\in[0,T],
\end{equation}
where \(h(t) := Q_F e^{\int_0^t f(u)\ud u}\) and \(k(t) := e^{-2\int_0^t f(u) \ud u}\).

We then restrict attention to controls that are Markovian in $(X^L_t, Y_t, Z_t)$, i.e.,
\begin{equation}
    \label{eqn:aug_info}
    u^L_t := u^L_{\mathrm{sub}}(t,X^L_t,Y_t,Z_t), \quad \text{where}\quad u^L_{\mathrm{sub}}:[0,T]\times \R^3\to\R,
\end{equation}
where the measurable function \(u^L_{\mathrm{sub}}\) satisfies \(\sup_{(t,\psi)\in [0,T]\times\R^3} |u^L_{\mathrm{sub}}(t,\psi)|/(1 + \|\psi\|)<\infty\), \(\psi:=(x,y,z)\in\R^3\). The corresponding augmented state process is denoted by $\Psi_t := (X_t^L, Y_t, Z_t)$. Clearly, the class of controls~\eqref{eqn:aug_info} is a subset of~\eqref{eqn:leader_info}; hence, optimal controls of this form may be suboptimal within the larger class~\eqref{eqn:leader_info}.

Under this subclass, Proposition~\ref{prop:leader_optimal_control} provides a semi-explicit solution for the leader's optimization problem. See Appendix~\ref{sec:appendix_b} for its proof.

\begin{proposition}\label{prop:leader_optimal_control}
For the FI maximization formulation~\eqref{eq:leader_dynamics} and~\eqref{eq:leader_fisher_objective_func},  the leader's optimal control $u^{L,\ast}$ (under the information structure~\eqref{eqn:aug_info}) is given by
\begin{equation}
    \label{eqn:leader_opt_ctrl}
    u^{L, \ast}_{\mathrm{sub}}(t,\psi)=-\frac1{R_L}\mathcal{B}\transpose (2 L_t \psi+M_t),
\end{equation}
where $L_t\in \mathbb{S}^{3\times3}$ and $M_t\in \mathbb{R}^3$ satisfy the ODEs    
\begin{align}
    \label{eq:leader_ODE_system_L}
    &\dot{L}_t + L_t\mathcal{A}(t) + \mathcal{A}(t)\transpose L_t - \frac{2}{R_L}L_t\mathcal{B}\mathcal{B}\transpose L_t + \mathcal{Q}(t) = 0,  \\
    \label{eq:leader_ODE_system_M}
    &\dot{M}_t + \mathcal{A}(t)\transpose M_t - \frac{2}{R_L}L_t\mathcal{B}\mathcal{B}\transpose M_t -Q_L F(t)e_1 = 0,
\end{align}
with terminal conditions \(L_T\) and \(M_T\).
Here,
\begin{align*}
    &\mathcal{A}(t) := \begin{bmatrix}
        A_L & 0 & 0\\ -h(t) & 0 & 0\\ 0 & k(t) & 0
    \end{bmatrix},\
    L_T := \begin{bmatrix}
        \frac{Q_{L, T}}{2} & 0 & 0 \\
        0 & -\frac{B_F^4\lambda_L}{\sigma_F^2 R_F^2} \int_0^T k(s) \ud s & \tfrac{B_F^4\lambda_L}{\sigma_F^2 R_F^2} \\
        0 & \tfrac{B_F^4\lambda_L}{\sigma_F^2 R_F^2} & 0
    \end{bmatrix},\ \mathcal{B} := B_Le_1,\\
    &\mathcal{Q}(t) := \diag\Big\{\frac{Q_L}{2},-\frac{B_F^4}{\sigma_F^2 R_F^2}\lambda_L k(t),0\Big\},\
    \mathcal{C} := \sigma_Le_1,\ M_T := -Q_{L, T} F(T)e_1,
\end{align*}
where \(e_1 = (1, 0, 0)\transpose \in\R^3\) denotes the first standard basis vector. 
\end{proposition}

    Denote by \(X^{L,*}\) the leader's optimal state process induced by \(u^{L,*}_{\mathrm{sub}}\) in equation~\eqref{eqn:leader_opt_ctrl}, which follows the dynamics:
    \begin{equation*}
        \ud X_t^{L,\ast} = \bigg( \Big(A_L - \frac{2B_L^2}{R_L}L_{11}(t)\Big)X_t^{L,\ast} + \int_0^t \phi(t,s) X_s^{L,\ast} \ud s - \frac{B_L^2}{R_L}M_1(t)\bigg) \ud t + \sigma_L \ud W^L_t,
    \end{equation*}
    where $\phi(t,s) := \frac{2 B_L^2}{R_L}h(s)\Big(L_{12}(t) + L_{13}(t)\int_s^t k(u)\ud u\Big)$ for $0\leq s\leq t\leq T$.
    The Volterra kernel \(\phi\) demonstrates the path-dependent nature of \(X^{L,*}\).

The existence of the leader's optimal control \(u^{L,*}_{\mathrm{sub}}\) depends on the well-posedness of the ODEs~\eqref{eq:leader_ODE_system_L}--\eqref{eq:leader_ODE_system_M}.
Because of the indefinite structure of the ODE governing $L_t$, global well-posedness is hard to achieve. 
Instead, we provide a local existence result in Theorem~\ref{thm:leader_ODE_system}, the proof of which is left to Appendix~\ref{sec:appendix_b}.

\begin{theorem}
\label{thm:leader_ODE_system}
    The ODE system~\eqref{eq:leader_ODE_system_L}--\eqref{eq:leader_ODE_system_M} admits a unique solution on \([0,T]\), provided that 
    \begin{equation*}
        T < \frac{1}{\sqrt{\beta q}} \arctan\left(\sqrt{\frac{q}{\beta }}\frac{1}{y_0}\right),
    \end{equation*}
   where the constants are defined as
    \begin{align*} 
    &q := \max\Big\{\Big|\frac{Q_L}{2} - |A_L| - \|h\|_\infty\Big|, \Big(\frac{B_F^4}{\sigma_F^2R_F^2}\lambda_L + 1\Big) \|k\|_\infty\Big\},\\
    &\beta := \max\Big\{\frac{2}{R_L}B_L^2 + |A_L|,\|h\|_\infty,\|k\|_\infty\Big\}, \quad y_0 := \max\Big\{\frac{Q_{L, T}}{2},\frac{B_F^4}{\sigma_F^2R_F^2}\lambda_L (\|k\|_1 + 1)\Big\}.
    \end{align*}
\end{theorem}

We remark that, the existence interval in Theorem~\ref{thm:leader_ODE_system} is far from maximal. Numerical experiments in Section~\ref{sec:numerics} indicate that the ODE system may remain well-posed beyond the theoretical bound.

\subsection{Extension to Multi-Period Interactions}\label{subsec:repeated_experiments}
As an extension of the Stackelberg game proposed above, we consider \(N\) consecutive episodes of leader-follower interactions, where the information accumulated across episodes enhances the inference of $M$.
Throughout, superscripts \((i)\) denote quantities within the \(i\)-th episode, where \(i\in[N]\).

The multi-period game unfolds over the time horizon \([0,NT]\), partitioned into \(N\) episodes of length \(T\).
In the first episode \([0,T]\), \(X^{L,*,(1)}\) and \(X^{F,*,(1)}\) are realized and the corresponding MLE \(\widehat M^{(1)}\) is calculated.
In the subsequent episode \([T,2T]\), both agents employ the same optimal strategies \(u^{L,*},\pi^{F,*}\) as before.
Their state processes, initialized as \(X^{L,*,(2)}_0 = X^{L,*,(1)}_T\) and \(X^{F,*,(2)}_0 = X^{F,*,(1)}_T\), evolve under independently resampled Brownian motions, producing a new estimate \(\widehat{M}^{(2)}\). This procedure repeats over \(N\) consecutive episodes, generating MLEs \(\widehat M^{(1)},\ldots, \widehat M^{(N)}\) (cf. equation~\eqref{eq:MLE_estimator}). Operationally, each period of duration \(T\) corresponds to the leader executing its control, observing the follower’s response, performing inference, and then initiating a new episode.

Each \(\widehat M^{(i)}\) contains information about \(M\), motivating the construction of a more efficient aggregated estimator:
\begin{equation}
    \label{eqn:period_MLE}
    \overline{M}_N :=  \sum_{i=1}^{N} w^N_i \widehat{M}^{(i)}, \quad w^N_i :=  \frac{\int_0^T g^2(t, X^{L,*}_{[(i-1)T, iT]}) \ud t}{\sum_{i=1}^N \int_0^T g^2(t, X^{L,*}_{[(i-1)T, iT]}) \ud t}.
\end{equation}
This estimator $\overline{M}_N$ remains unbiased and is more efficient than any single \(\widehat{M}^{(i)}\):
\begin{equation*}
    \Var[\overline{M}_N] = \frac{\sigma_F^2R_F^2}{B_F^4} \E\Bigg[\frac{1}{\sum_{i=1}^N \int_0^T g^2(t, X^{L,*}_{[(i-1)T, iT]}) \ud t}\Bigg]\leq \Var[\widehat M^{(i)}],\ \forall i\in[N].
\end{equation*}

Additionally, the estimator \(\overline{M}_N\) admits an online update scheme with a natural stopping rule.
Given the current estimate \(\overline{M}_N\) and the latest observation \(\widehat{M}^{(N+1)}\), 
\begin{align*}
    \overline{M}_{N+1} = \frac{P_N}{P_{N+1}} \overline M_N  + \Big(1 - \frac{P_N}{P_{N+1}}\Big) \widehat M^{(N+1)},\quad \Var[\overline{M}_{N+1}] = \frac{\sigma_F^2R_F^2}{B_F^4} \E\Bigg[\frac{1}{P_{N+1}}\Bigg],
\end{align*}
where the precision $P_N := \sum_{i=1}^N \int_0^T g^2(t, X^{L,*}_{[(i-1)T, iT]}) \ud t$ can be incrementally updated after each episode. This formulation allows the leader to implement a stopping rule for \(N\), ensuring that the total estimation variance remains below a prescribed threshold.

{Importantly, in both formulations~\eqref{eq:leader_var_objective_func}--\eqref{eq:leader_fisher_objective_func}, the leader’s optimal policy does not depend on the true value of \(M\), since both the variance of \(\widehat M\) and the Fisher information are independent of \(M\) (cf. Proposition~\ref{prop:MLE_calc}).
Consequently, in the multi-period extension considered above, the leader's optimal information-revealing control $u^{L, \ast}$ within each single period does not benefit from increasingly accurate estimates of \(M\), and can therefore be reused.
The multi-period extension can thus be viewed as a repeated-experiment setting, where the estimation of \(M\) is progressively refined from the accumulated observations of the follower's state trajectory.}

\section{Numerical Algorithms}\label{sec:numerics}
In this section, we develop machine learning algorithms to solve the strategic inference model, providing an intuitive illustration of the leader-follower interactions. Specifically, 
Section~\ref{subsec:optimal_control_comparison} compares the leader’s optimal state and control trajectories under two objectives, while Section~\ref{subsec:estimator_behavior_validation} evaluates the performance of the MLE $\widehat{M}$ and its multi-period extension $\overline{M}_N$, both computed at the final estimation stage.

Due to the path-dependent nature of the leader's objective, we parameterize \(u^L\) using a recurrent neural network (RNN).
Unlike feedforward architectures, RNNs are naturally suited for sequential data and can capture temporal dependencies of variable length.
{RNNs encode past information by maintaining an internal cell state that is recursively updated over time. 
At each time step, the cell state aggregates information from the previous cell state and current inputs, allowing the network to represent the entire input history in a compressed form.}
Specifically, we adopt a long short-term memory (LSTM) network~\cite{hochreiter1997long}, which mitigates vanishing gradients and effectively models long-term dependencies {via gating mechanisms that regulate how past information is retained, forgotten, or updated. In LSTM architectures, the recurrence mechanism is implemented through a cell state that carries long-term information, while the hidden state produces the network output and depends on the input history up to the current time.}

To solve leader's both optimization formulations, we apply the \emph{direct parameterization (DP)} algorithm \cite{han2016deep,HanHu2021}.
Within each of the \(N_{\mathrm{epoch}}\) training epochs, the algorithm simulates \(N_{\mathrm{batch}}\) state trajectories of the leader to approximate the expected cost~\eqref{eq:leader_var_objective_func}--\eqref{eq:leader_fisher_objective_func}, which serves as the loss function for neural training.

\smallskip
\noindent\textbf{Numerical Implementations\footnote{All experiments are implemented in \texttt{PyTorch} and are executed on an Nvidia GeForce RTX 2080 Ti GPU. The code used in this work is publicly available at \url{https://github.com/dralston78/Strategic-Inference-Stackelberg-Games}
}.} We discretize the time horizon \([0,T]\) into \(N_T\) sub-intervals of equal length \(h:=T/N_T\), and denote the discrete grid by \(\Delta:=\{kh:k\in\{0,1,\ldots,N_T-1\}\}\). We approximate the leader's state~\eqref{eq:leader_dynamics}  by the Euler scheme:
\begin{equation*}
     \tilde{X}^L_{t+h} = \tilde{X}^L_t + \Big(A_L \tilde{X}_t^L + B_L u^L_{\mathrm{NN}}(t,\tilde{X}^L_{[0,t]}) \Big) h + \sigma_L \sqrt{h}\xi^L_t,\ \forall t\in\Delta,
\end{equation*}
where \(\{\xi^L_t\}\overset{\mathrm{i.i.d.}}{\sim} \mathcal{N}(0,1)\), and $u^L_{\mathrm{NN}}$ denotes the network output. 

At each time step \(t\in\Delta\), the LSTM receives as inputs {a ten-dimensional vector consisting of} the simulated leader's state \(\tilde{X}_t^L\), a weighted vector of the three most recent states \((\gamma^3 \tilde{X}_{t-3h}^L, \gamma^2 \tilde{X}_{t-2h}^L, \gamma \tilde{X}_{t-h}^L)\) with a decay factor \(  \gamma\in(0,1] \), their unweighted counterpart \((\tilde{X}_{t-3h}^L, \tilde{X}_{t-2h}^L, \tilde{X}_{t-h}^L)\), and timestamps \((t/T, \sin(\pi t/T), \cos(\pi t/T))\). {This input vector, together with the previous hidden state, passes through the gating mechanisms of the LSTM \cite[equation~(7)]{HanHu2021} to produce the cell state, which subsequently determines the updated hidden state.} 
The updated hidden state is then passed through two fully connected layers, producing network output \(u^L_{\mathrm{NN}}(t,\tilde{X}^L_{[0,t]}) \in \R \), which approximates the control value at time \(t\).
To ensure stability of the hidden state, we include a warm-up step by setting \( \tilde{X}_{-3h}^L = \tilde{X}_{-2h}^L = \tilde{X}_{-h}^L = \tilde{X}_0^L \) at three fictitious negative times \(-3h, -2h, -h \). 
This architecture follows the parameterization in~\cite{HanHu2021}, which has been validated to provide robust numerical performance.

\smallskip
\noindent\textbf{Hyperparameters.}
The LSTM layer uses a hidden size of \(64\),  followed by two fully connected layers of widths $64$ and $32$. The learning rate is initialized at $0.001$ and reduced by a factor of \(0.5\) after \(1500\) and \(5000\) epochs, respectively. All network parameters are updated using the Adam optimizer. Unless otherwise specified, the following hyperparameters are used: \(N_T = 50,\ N_{\mathrm{batch}} = 512,\ N_{\mathrm{epoch}} = 10000\). For Figure~\ref{fig:primary_info_tradeoff}, we use a finer time discretization \(N_T = 500\) for better visualizations of the tracking behavior.

\smallskip
\noindent\textbf{Model Parameters.} 
All numerical results are based on the following configuration:
\begin{align*}
    &T = 0.5,\quad A_L=A_F=-1.0,\quad B_L=B_F=1.0,\quad \sigma_L=\sigma_F=0.1,\quad X^L_0 = X^F_0 = 0.1,\\
    &Q_L = Q_F = Q_{L, T} = R_L = R_F = \lambda_F = 1.0,\quad M = 1.0,\quad F(t) = 0.1 \sin(2 \pi t/T).
\end{align*}
The parameter $\lambda_L$, which quantifies the strength of strategic inference, is varied across experiments and specified separately for each case.

 Although this parameter set may violate the condition in Theorem~\ref{thm:leader_ODE_system}, numerical verification confirms that the ODE system~\eqref{eq:leader_ODE_system_L}–\eqref{eq:leader_ODE_system_M} remains well-posed on \([0,T]\).

\subsection{Variance Minimization \emph{vs.} FI Maximization}\label{subsec:optimal_control_comparison}

We numerically compare the leader's optimal states and controls obtained by solving the variance minimization formulation~\eqref{eq:leader_var_objective_func} and the FI maximization formulation~\eqref{eq:leader_fisher_objective_func}, highlighting their connection and trade-offs.

\begin{figure}[htbp]
    \centering
    \includegraphics[width=1.0\linewidth]{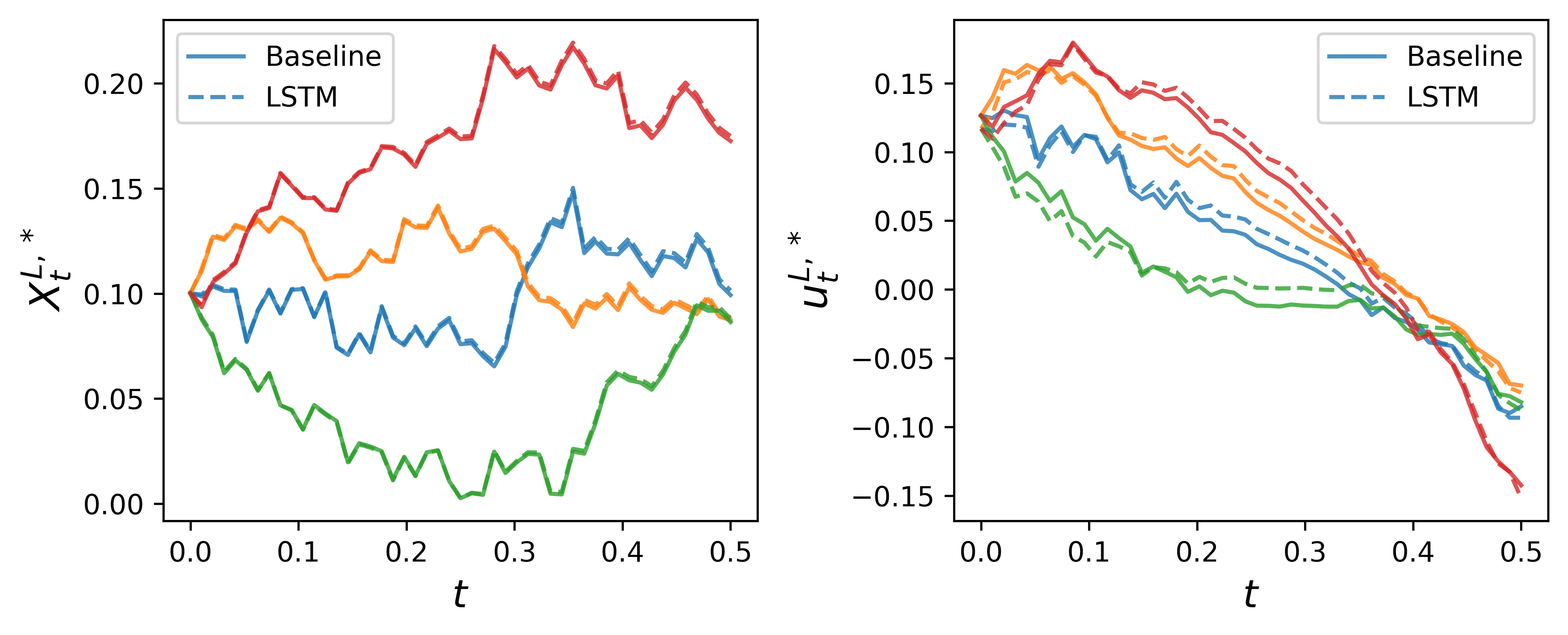}
    \vspace*{-1.5em}
    \caption{
    Leader's optimal state trajectories (left) and control trajectories (right) for the FI maximization formulation~\eqref{eq:leader_fisher_objective_func} with \(\lambda_L = 0.5\).  
    Solid lines represent baseline trajectories from Proposition~\ref{prop:leader_optimal_control}, while dashed lines denote LSTM-based approximations.  Different colors correspond to different sample paths of $W^L$. {The final training loss of LSTM network is $-0.004768$, while \eqref{eq:leader_fisher_objective_func} under \(u^{L,*}_{\mathrm{sub}}\) gives a value of $-0.004823$.}
    }
    \label{fig:FIM_NN_analy}
\end{figure}

As a sanity check for our numerical implementations, Figure~\ref{fig:FIM_NN_analy} compares the leader's optimal trajectories for the FI maximization formulation~\eqref{eq:leader_fisher_objective_func} generated by \(u^{L,*}_{\mathrm{sub}}\) in Proposition~\ref{prop:leader_optimal_control}, with the LSTM-approximated ones under identical sample paths of $W^L$.  
{We observe a final training loss of $-0.004768$, while \eqref{eq:leader_fisher_objective_func} under \(u^{L,*}_{\mathrm{sub}}\) gives a value of $-0.004823$.  The small difference between the training loss and the analytical cost, as well as the close alignment between controls depicted in Figure~\ref{fig:FIM_NN_analy}, suggests that \(u^{L,*}_{\mathrm{sub}}\) remains near-optimal under the general information structure~\eqref{eqn:leader_info}, even though its optimality has only been established theoretically within the restricted subclass~\eqref{eqn:aug_info}.}

\begin{figure}[htbp]
    \centering
    \includegraphics[width=1.0\linewidth]{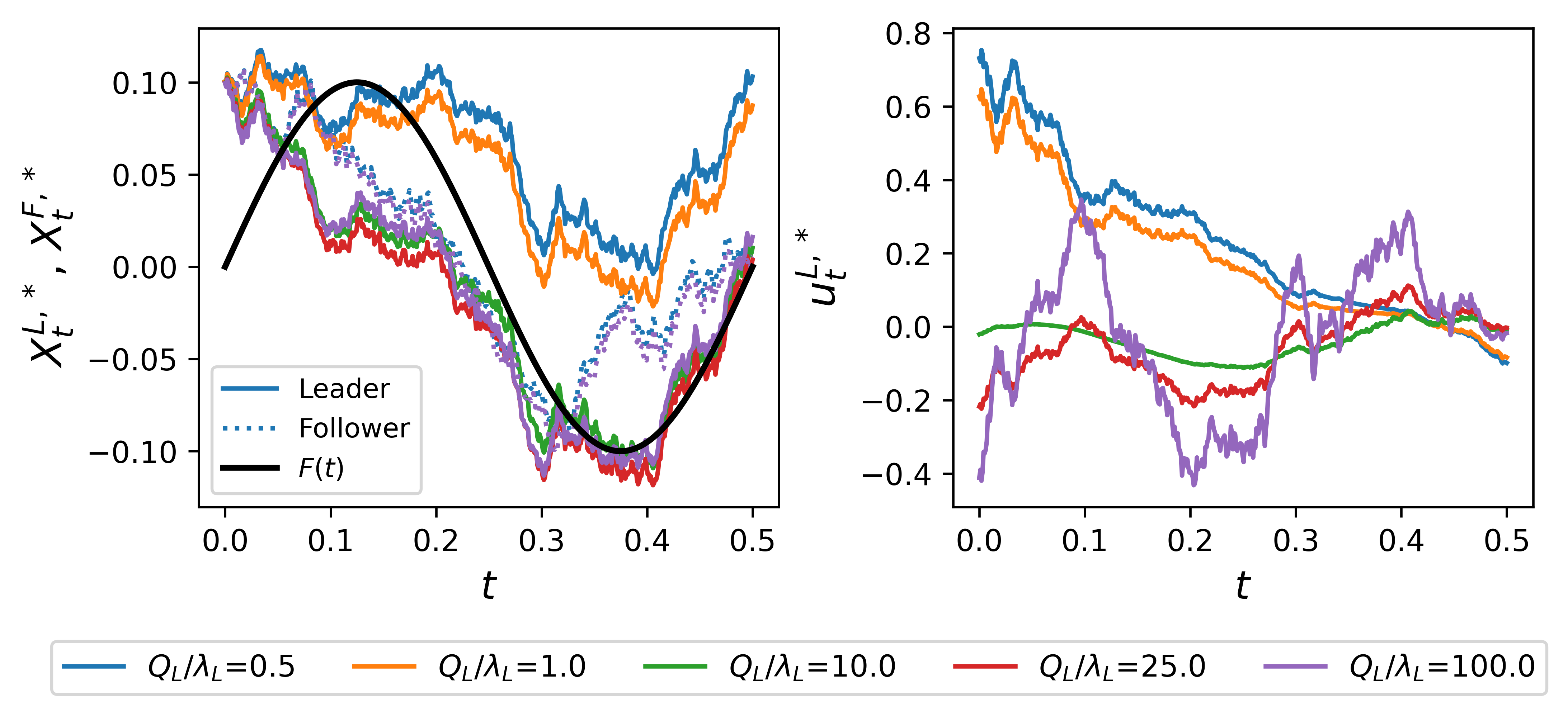}
    \vspace*{-1.5em}
    \caption{
    Leader/follower's optimal state trajectories (left) and leader's optimal control trajectories (right) under the FI maximization formulation~\eqref{eq:leader_fisher_objective_func} under different values of \(Q_L/\lambda_L\).
    Solid lines denote the leader's trajectories, and dotted lines indicate the follower's trajectories.  All leader's trajectories are generated under the baseline control \(u^{L,*}_{\mathrm{sub}}\) from Proposition~\ref{prop:leader_optimal_control}, sharing the same sample path of $W^L$. {The follower's trajectory is plotted only for $Q_L/\lambda_L = 0.5$ due to minimal variation across cases (cf. Section~\ref{subsec:optimal_control_comparison}).}
    A finer time discretization $N_T=500$ is adopted for better visualizations of the tracking behavior. 
    }
  \label{fig:primary_info_tradeoff}
\end{figure}

Figure~\ref{fig:primary_info_tradeoff} illustrates the trade-off between the leader's fulfillment of the primary task \textit{vs.} strategic inference, characterized by the ratio $Q_L/\lambda_L$. A large ratio indicates stronger emphasis on tracking \(F(t)\), while a small ratio prioritizes information extraction about \(M\).
As the ratio increases from \(0.5\) to \(1\), \(10\), \(25\), \(100\), the leader's optimal state trajectories increasingly align with the target $F(t)$, while the Fisher information \(I(M)\) (averaged over \(10000\) sample paths) decreases from \(0.146\) to \(0.112\), \(0.015\), \(0.004\), \(0.001\).
{On the follower's side, \(X^{F,*}_t\) exhibits little change, as the tracking coefficient $Q_F = 1.0$ provides only a mild incentive to track the leader.}

\begin{figure}
    \centering
    \includegraphics[width=1.0\linewidth]{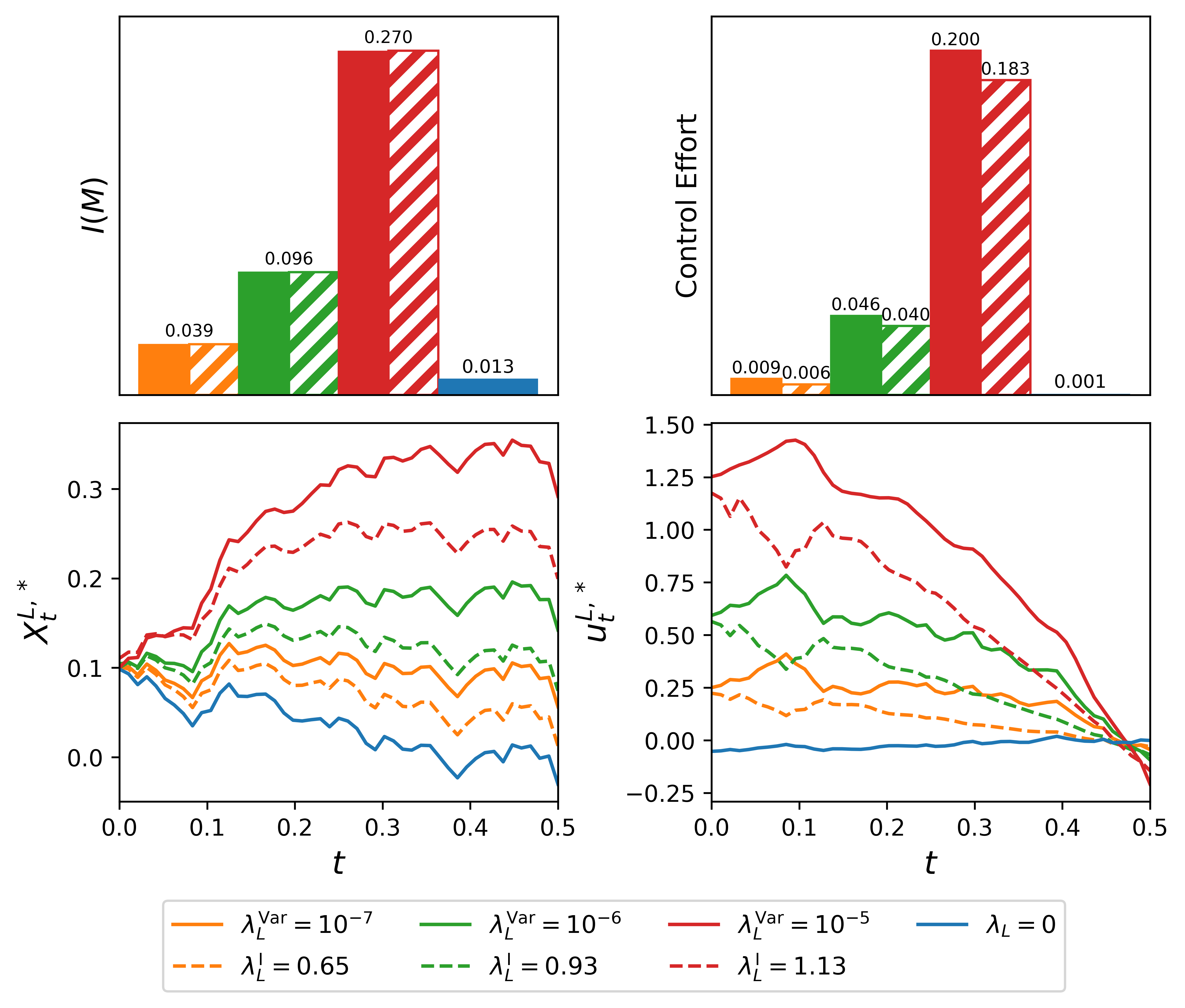}
    \vspace*{-1.5em}
    \caption{Comparisons of FI (top left), control effort (top right), and optimal trajectories (bottom) under the variance minimization~\eqref{eq:leader_var_objective_func} and FI maximization~\eqref{eq:leader_fisher_objective_func} formulations.
    Lines/bars in the same color correspond to intensity tuples \((\lambda_L^\Var, \lambda_L^{I})\) in~\eqref{eqn:tuple_intensity}, for which $I(M)$ is approximately aligned.
    Each value in the top panels is estimated from \(10000\) sample paths. 
    In the bottom panels, solid and dashed lines represent optimal trajectories under variance minimization~\eqref{eq:leader_var_objective_func} and FI maximization~\eqref{eq:leader_fisher_objective_func}, respectively, sharing the same sample path of \(W^L\).
    }
    \label{fig:Var_FIM_cnrl_comparison}
\end{figure}

Figure~\ref{fig:Var_FIM_cnrl_comparison} compares the two formulations $J_\Var^L$ and $J_I^L$ across different values of \(\lambda_L\). Shared quantities are labeled with superscripts $\Var$ and $I$ for clarity.
To ensure fair comparison, we choose tuples:
\begin{equation}
    \label{eqn:tuple_intensity}
    (\lambda_L^\Var, \lambda_L^{I}) \in \{(10^{-7},0.65),(10^{-6},0.93),(10^{-5},1.13)\},
\end{equation}
under which the resulting Fisher information $I(M)$ is closely aligned under optimal controls.
With identical levels of $I(M)$, the FI maximization formulation uses slightly less control effort, measured by $\E\int_0^T \big(u_t^{L,*}\big)^2 \ud t$.
Notably, the two formulations generate distinct control profiles, and the required intensities differ by several orders of magnitude \(\lambda_L^\Var \ll \lambda_L^{I}\).
Moreover, the optimal control for $J_I^L$ is more sensitive to variations in $\lambda_L$; whereas $J_\Var^L$ requires significantly larger adjustments in $\lambda_L$ to produce comparable variations in $I(M)$.

\subsection{Performance for $\widehat M$}\label{subsec:estimator_behavior_validation} We evaluate the performance of the estimators of \(M\) derived at the final inference stage of the game. Figure~\ref{fig:single_period} compares the conditional bias and conditional variance of the MLE~\eqref{eq:MLE_estimator} among different values of \(\lambda_L\), using a fixed leader's trajectory \(x^{L,*}\) generated under the FI maximization formulation~\eqref{eq:leader_fisher_objective_func}. The left panel confirms the conditional unbiasedness of \(\widehat M\), consistent with the Monte Carlo convergence rate. 
The right panel shows that the conditional variance decreases sharply with increasing $\lambda_L${, with a faster decay for larger values of $\lambda_L$}, demonstrating the effectiveness of strategic inference in improving estimation precision.

\begin{figure}
    \centering
    \includegraphics[width=1.0\linewidth]{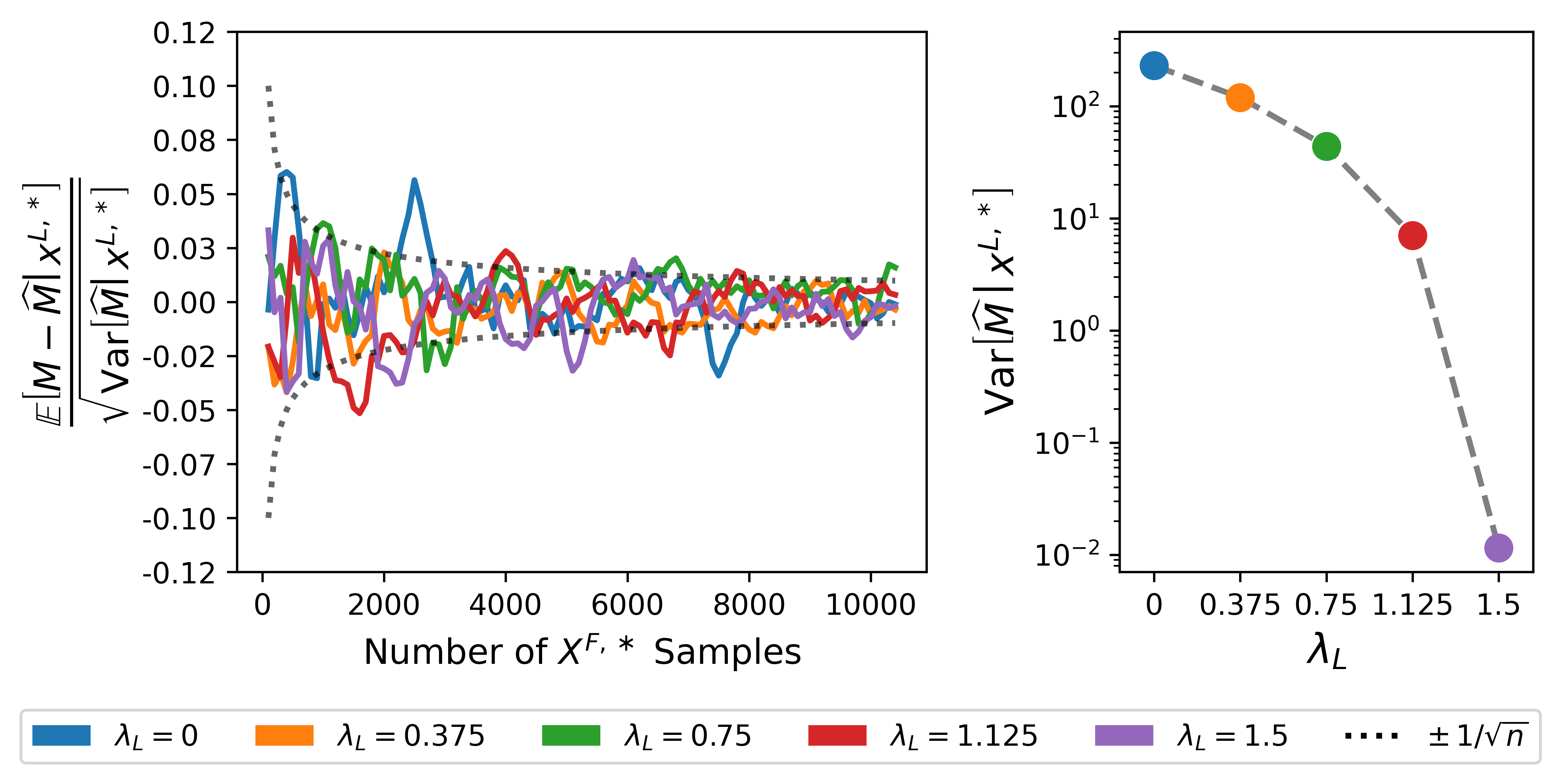}
    \vspace*{-1.5em}
    \caption{
    Conditional bias (left) and conditional variance (right) of the MLE \(\widehat M\) under different values of \(\lambda_L\).  
    The leader's trajectories \(x^{L,*}\) are generated using the FI maximization formulation~\eqref{eq:leader_fisher_objective_func} with the control \(u^{L,*}_{\mathrm{sub}}\) from Proposition~\ref{prop:leader_optimal_control}, sharing the same sample path of \(W^L\). A three-point moving average is applied to the bias plot for readability.
    }
    \label{fig:single_period}
\end{figure}

\begin{figure}
    \centering
    \includegraphics[width=1.\linewidth]{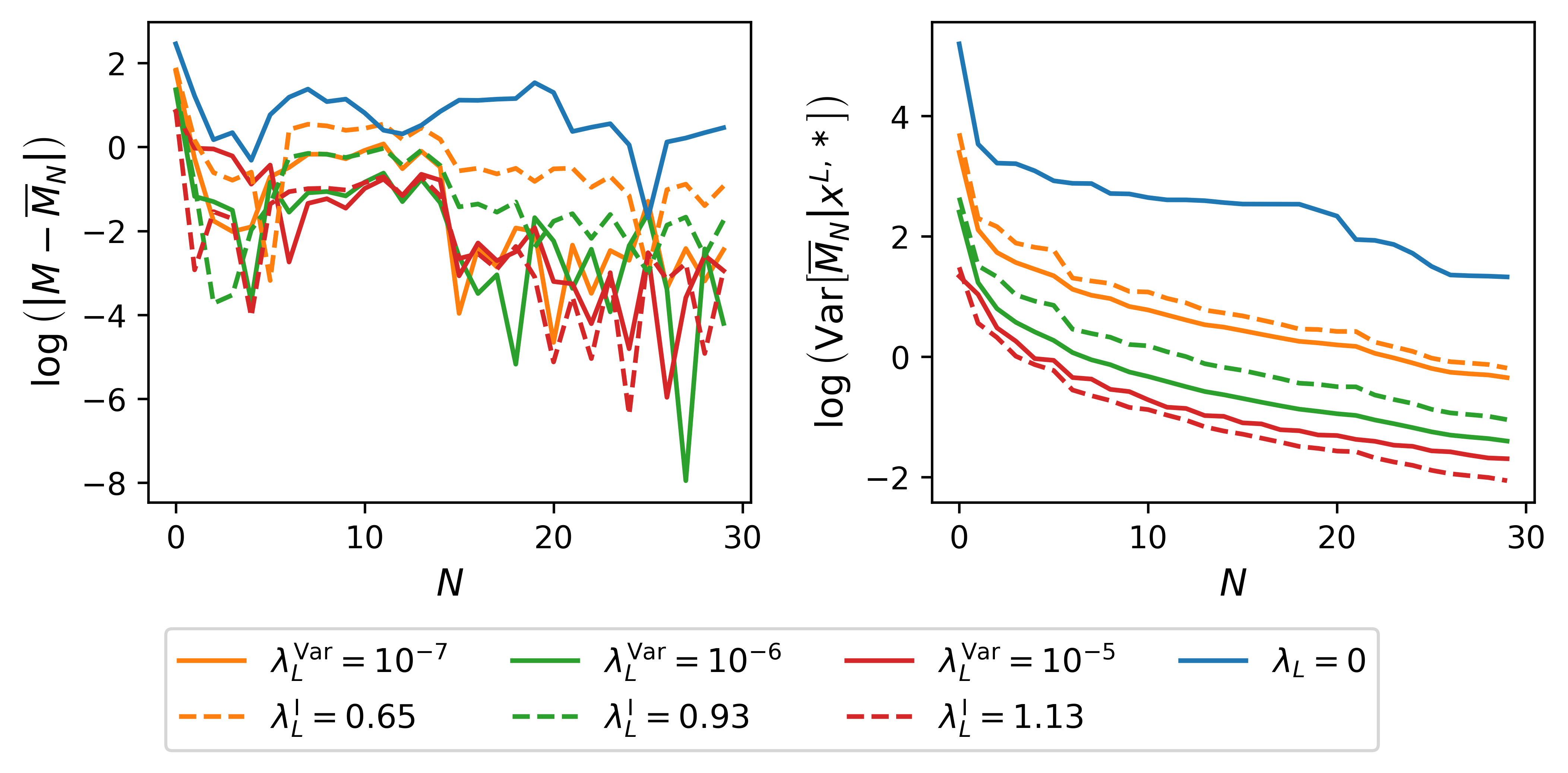}
    \vspace*{-1.5em}
    \caption{
    Inference error (left) and conditional variance (right) of the multi-period estimator~\(\overline{M}_N\)~\eqref{eqn:period_MLE} under the extended formulations~\eqref{eq:leader_var_objective_func}--\eqref{eq:leader_fisher_objective_func}. 
    Solid (resp.~dashed) lines represent trajectories generated by the optimal control from the variance minimization~\eqref{eq:leader_var_objective_func} ((resp.~FI maximization~\eqref{eq:leader_fisher_objective_func}) formulation. 
    Lines with the same color are associated with tuples \((\lambda_L^\Var, \lambda_L^{I})\) in~\eqref{eqn:tuple_intensity}, for which $I(M)$ is aligned across formulations.
    The leader's (resp. follower's) trajectories share the same sample path of \(W^L\) (resp. $W^F$).
    Both quantities are plotted on a log scale for clarity.
    }
    \label{fig:multi_period}
\end{figure}

Finally, Figure~\ref{fig:multi_period} evaluates the inference error and conditional variance of the multi-period estimator \(\overline{M}_N\) (cf. Section~\ref{subsec:repeated_experiments}).
For a fair comparison, we use the intensity tuples from~\eqref{eqn:tuple_intensity} to align the FI under optimal controls.  
On the left panel, each MLE $\overline{M}_N$ is generated from a single leader and follower trajectory within each period.
We observe that larger values of \(\lambda_L\) substantially reduce the inference error, by several orders of magnitude, whereas in the absence of information acquisition ($\lambda_L = 0$),  the error remains high even after $N=30$ periods.
The variance of $\overline{M}_N$ decreases monotonically with the number of periods $N$ and the intensity \(\lambda_L\), consistent with faster Monte Carlo convergence induced by a higher level of strategic inference.

\section{Conclusions and Future Work}\label{sec:conclusions}
In this paper, we proposed a continuous-time Stackelberg game to model strategic inference between a leader and a follower. Anticipating the follower’s optimal tracking response, the leader designs its control to simultaneously accomplish a primary task and extract maximal information about the follower's latent parameter \(M\). This game-theoretic formulation offers a new perspective on inverse learning via stochastic experimental design.

Several directions remain open for future research. One natural extension is to higher-dimensional settings with discrete-time observations, {where the follower’s randomization through \(\lambda_F\) may affect the leader’s ability to perform inference, as well as} games involving multiple homogeneous agents. In addition, inferring other model parameters such as \(Q_F\), \(R_F\), or jointly estimating them with \(M\), would broaden the applicability of our framework. {In the present work, the follower is assumed to know its model and to solve the entropy-regularized LQ tracking problem in a model-based way. An interesting extension would be to relax this assumption and let the follower learn an optimal randomized tracking policy via model-free RL, based only on observed state and cost trajectories; see, for instance, recent works on entropy-regularized LQ control without full model knowledge~\cite{szpruch2024optimal,giegrich2024convergence,guo2026fast}. In such a setting, the leader would still act strategically to infer the latent parameter \(M\), but now from the behavior of a learning agent rather than a fully informed controller, leading to a coupled ``learning follower \emph{vs.} inferring leader'' system.}

{Currently our analysis focuses on a sequential ``observe–then–track'' information structure, in which the follower observes the entire leader trajectory on $[0,T]$ before solving its tracking problem. In more standard dynamic formulations with simultaneous system evolution, the follower would only observe $(X_t^F,X^L_{[0,t]})$ (or $(X_t^F,X_t^L)$) at time $t$, leading either to a dynamic Stackelberg game with partial observations or to a dynamic Nash equilibrium in which both agents choose feedback strategies that are mutual best responses. Analyzing such online, partially informed variants and comparing their performance with our full-path benchmark is an interesting direction.}

Beyond these, our current model leverages a linear-quadratic-Gaussian structure, under which the Fisher information is independent of \(M\), allowing for tractable optimization. Extending to more general dynamics or distributional families would enhance realism but requires fundamentally new techniques.
{The adaptive extension of the multi-period estimation framework is another appealing direction for future research.
As a natural approach, the leader may leverage information extracted from previous periods to design information-revealing controls that more effectively uncover the follower's latent intent $M$. 
Such a development would require constructing a new global estimator and a corresponding objective for the leader, aimed at improving estimation efficiency across all periods.} 
Another important direction is to incorporate adversarial deception: when the leader’s observations are noisy~\cite{hooker2015control}, or when the follower deliberately distorts its trajectory~\cite{zhou2025adversarial,zhou2025integrating}, inference becomes significantly harder. Finally, integrating agent heterogeneity, e.g., learning over networks with diverse preferences~\cite{caines2021graphon,hu2025finite}, would move the model closer to complex real-world environments.

\appendix

\section{Proofs of Proposition~\ref{prop:follower_optimal_control} and Theorem~\ref{thm:follower_ODE}}\label{sec:appendix_a}

This appendix provides proofs for Proposition~\ref{prop:follower_optimal_control} and Theorem~\ref{thm:follower_ODE}, both arising from the follower’s optimization problem discussed in Section~\ref{subsec:followers_problem}.

\begin{proof}[Proof of Proposition~\ref{prop:follower_optimal_control}]
By the dynamic programming principle (DPP), the follower's value function $V(t, x)$ satisfies the Hamilton-Jacobi-Bellman (HJB) equation
\begin{multline*}
        \partial_t V+ 
    \inf_{\pi_t^F}\Bigg\{\int_{\mathbb{R}} \Bigl[ B_F y \,\partial_xV + \frac{R_F}{2}y^{2}+\lambda_F\log\pi_t^F(y) \Bigr]\pi_t^F(y) \ud y \Bigg\}
       \\  + \frac{Q_F}{2}(x - Mx_t^L)^2 + A_F x \,\partial_x V + \frac{\sigma^{2}_F}{2}\partial_{xx}V = 0,
\end{multline*}
with terminal condition $V(T,x) = 0$.

Solving for the infimum (setting the first variation in $\pi_t^F$ as zero) yields
\begin{equation*}
    \pi^{F,*}(y\mid t,x) \propto e^{\frac{1}{\lambda_F}\bigl(-\frac{R_F}{2}y^{2}-B_Fy \,\partial_x V \bigr)},
\end{equation*}
which implies that $\pi^{F,*}(\cdot\mid t,x)=\mathcal N\Bigl(-\frac{B_F}{R_F}\partial_x V,\frac{\lambda_F}{R_F}\Bigr)$ is Gaussian.
Consider the quadratic ansatz $V(t,x)=a_tx^{2}+b_tx+c_t$, where $a,b,c$ are deterministic measurable functions. 
Plugging \(\pi^{F,*}\) back into the HJB equation and collecting coefficients yield the ODE system~\eqref{eq:follower_ODE_sys_a}, with \(c_t\) satisfying
\begin{equation*}
    \dot{c}_t = -\frac{Q_F M^2}{2}( x_t^L)^{2} + \frac{B_F^{2}}{2R_F}b^{2}_t
    -\sigma_F^{2}a_t
    + {\frac{\lambda_F}{2}\log\Bigl(\frac{2\pi \lambda_F}{R_F}\Bigr),}
\end{equation*}
with a given terminal condition \(c_T = 0\).
Note that the solution \(c_t\) exists whenever both \(a_t\) and \(b_t\) exist.
Plugging the ansatz into \(\pi^{F,*}\) yields the optimal policy~\eqref{eqn:F_opt_pi}, which concludes the proof.
\end{proof}

\begin{proof}[Proof of Theorem~\ref{thm:follower_ODE}]
    Since \(R_F,Q_F\) are both positive and $a_T$ is non-negative, by \cite[Theorem~4.1.6]{abou2012matrix}, \(a_t\) exists on \([0,T]\), for any \(T>0\).
    Given \(a_t\), the solution to the linear ODE in \(b_t\) with time-varying coefficients globally exists. This concludes the proof.
\end{proof}

\section{Proofs of Propositions~\ref{prop:MLE_calc},~\ref{prop:leader_optimal_control}, and Theorem~\ref{thm:leader_ODE_system}}\label{sec:appendix_b}

This appendix provides proofs for Proposition~\ref{prop:MLE_calc}, Proposition~\ref{prop:leader_optimal_control}, and Theorem~\ref{thm:leader_ODE_system}, all arising from the leader’s side discussed in Section~\ref{sec:leader_and_MLE}.

\begin{proof}[Proof of Proposition~\ref{prop:MLE_calc}]

Plugging dynamics~\eqref{eqn:opt_state} into the MLE~\eqref{eq:MLE_estimator} yields
\begin{equation}
    \label{eqn:MLE_form}
    \widehat{M} = M - \frac{\sigma_F\int_0^T g(t,X^L) \ud W_t^F}{\frac{B_F^2}{R_F}\int_0^T  g^2(t,X^L) \ud t}.
\end{equation}
By the Itô isometry, the conditional mean, variance, and FI are given by
\begin{align*}
    &\E[\widehat{M}|X^L] = M,\qquad  \Var[\widehat{M}|X^L] = \frac{\sigma_F^2}{\frac{B_F^4}{R_F^2}\int_0^T g^2(t, X^L) \ud t},\\
    &I(M|X^L) := \E \Big[(\partial_M \ell(X^{F,*},X^L;M))^2\Big|X^L\Big] = \frac{B_F^4}{\sigma_F^2R_F^2} \int_0^T g^2(t,X^L)\ud t.
\end{align*}
Taking expectations on both sides and using the conditional decomposition of variance yield equations~\eqref{eqn:MLE_var}--\eqref{eqn:MLE_FI}, which concludes the proof.
\end{proof}

\begin{proof}[Proof of Proposition~\ref{prop:leader_optimal_control}]
The definitions of $Y$ and $Z$~\eqref{eqn:Y_Z} yield the state dynamics:
$d\Psi_t=\mathcal{A}(t)\Psi_t \ud t + \mathcal{B}u_t^L \ud t + \mathcal{C} \ud W^L_t$. 
By equations~\eqref{eqn:g} and~\eqref{eqn:Y_Z},
\[
\int_{0}^{T}g^{2}(t,X^L)\ud t=Y_T^2 \int_0^T k(t)\ud t + \int_0^T k(t) Y_t^2 \ud t -2 Y_T Z_T.
\]
Combining with the primary task~\eqref{eqn:leader_primary}, the expected cost~\eqref{eq:leader_fisher_objective_func} becomes
\begin{multline*}
    J^L_I(u^L)=\E\Big[ \int_{0}^{T}\Big(\Psi_t\transpose \mathcal{Q}(t)\Psi_t - Q_L F(t)e_1\transpose \Psi_t + \frac{Q_L}{2}[F(t)]^2 + \frac{R_L}{2}(u_t^L)^{2}\Big)\ud t\\ + \Psi_T\transpose L_T\Psi_T - Q_{L,{T}}F(T)e_1\transpose \Psi_T +\frac{Q_{L,T}}{2}[F(T)]^2 \Big].
\end{multline*}

For this Markovian control problem, by the DPP, the leader's value function $V(t,\psi)$ satisfies the HJB equation 
\begin{multline*}
    \partial_t V + \inf_{u}\Big\{(\partial_\psi V)\transpose(\mathcal{A}(t)\psi+\mathcal{B}u) +\frac{R_L}{2}u^{2}\Big\}+\frac{1}{2} \mathrm{Tr} \big( \mathcal{C}\mathcal{C}\transpose \partial_{\psi\psi}V) + \psi\transpose \mathcal{Q}(t)\psi\\-Q_LF(t)e_1\transpose \psi + \frac{Q_L}{2}[F(t)]^2=0,
\end{multline*}
with a given terminal condition $V(T,\psi)=\psi\transpose L_T\psi -Q_{L, T}F(T) e_1\transpose \psi + \frac{Q_{L,T}}{2}[F(T)]^2.$ 

Solving for the infimum yields $u^{L, \ast}_{\mathrm{sub}}(t,\psi) =-\frac{1}{R_L}\mathcal{B}\transpose \partial_\psi V $.
Consider the ansatz $V(t,\psi) = \psi\transpose L_t \psi + M_t\transpose \psi + N_t$, where $L_t \in \mathbb{S}^{3\times3}$, $M_t \in \R^3$, $N_t \in \R$ are deterministic measurable functions.
Plugging back into the HJB equation and collecting coefficients yields the ODE system~\eqref{eq:leader_ODE_system_L}--\eqref{eq:leader_ODE_system_M}, with 
\begin{equation*}
    \dot{N}_t-\frac{1}{2R_L}\big(\mathcal{B}^\top M_t\big)^2 + \mathcal{C}\transpose L_t \mathcal{C} + \frac{Q_L}{2}[F(t)]^2 = 0,
\end{equation*}
given the terminal condition \(N_T = \frac{Q_{L, T}}{2} [F(T)]^2\).
Note that the solution \(N_t\) exists whenever both \(L_t\) and \(M_t\) exist.
Plugging the ansatz into \(u^{L,*}_{\mathrm{sub}}\) yields~\eqref{eqn:leader_opt_ctrl}, which concludes the proof.
\end{proof}

\begin{proof}[Proof of Theorem \ref{thm:leader_ODE_system}]
    Given the existence of \(L_t\), \(M_t\) satisfies a linear ODE with time-varying coefficients, whose solution uniquely exists.
    Hence, it suffices to prove that the solution $L_t$ uniquely exists on \([0,T]\), which is equivalent to establishing \textit{a priori} bounds for \(L_t\).

    We first derive a global upper bound for \(L_t\).
    Consider \(X(t)\) such that $\dot{X} = - \mathcal{A}\transpose(t) X-X\mathcal{A}(t) - \mathcal{Q}(t)$, with a given terminal condition $X(T) = L_T$.
    Since it is a linear ODE, \(X(t)\) exists on \([0,T]\).
    By the comparison principle of Riccati equations \cite[Theorem 4.1.4]{abou2012matrix}, we get $L_t \leq X(t)$.
    
    Next, we establish a lower bound for \(L_t\).  
    Define matrices
    \begin{align*}
        Q_Y &:= \diag\Big\{\tfrac{Q_L}{2} - |A_L| - \|h\|_\infty,-\Big(\frac{B_F^4}{\sigma_F^2R_F^2}\lambda_L + 1\Big) \|k\|_\infty,0\Big\},\\
        S_Y &:= \diag\Big\{\frac{2}{R_L}B_L^2 + |A_L|,\|h\|_\infty,\|k\|_\infty\Big\},
    \end{align*}
    and let $Y(t)$ satisfy $\dot{Y} = - Q_Y + Y S_Y Y$, with a given terminal condition $Y(T) = L_T$. 
    Since diagonally dominant symmetric matrices with positive diagonal entries are positive semi-definite (as a corollary of the Gershgorin circle theorem),
     \begin{equation*}        
        \begin{bmatrix}
            \mathcal{Q}(t) & \mathcal{A}\transpose(t)\\
            \mathcal{A}(t) & -\frac{2}{R_L}\mathcal{B}\mathcal{B}\transpose
        \end{bmatrix} \geq
        \begin{bmatrix}
            Q_Y & 0\\
            0 & -S_Y
        \end{bmatrix},\ \forall t\in[0,T].
    \end{equation*}
    By the comparison principle of Riccati equations \cite[Theorem 4.1.4]{abou2012matrix}, if $Y(t)$ exists on \([0,T]\), then $L_t\geq Y(t)$ provides a lower bound. 
    
    At this point, we only need to show that $Y(t)$ exists on \([0,T]\). Consider \(y(t)\) such that $\dot{y} = \beta y^2 + q$ with a given initial condition $y(0)=y_0$ (recall the definitions of constants in Theorem~\ref{thm:leader_ODE_system}).
    Under the restrictions on \(T\), \(y(t)\) exists on \([0,T]\), which follows from \cite{papavassilopoulos2003existence}.
    Since \(\|Q_Y\|\leq q\), \(\|S_Y\|\leq \beta\), and \(\|L_T\|\leq y_0\), by \cite[Corollary 6.3]{hale2009ordinary} \(\|Y(T-t)\|\leq y(t),\ \forall t\in[0,T]\).
    The existence of \(Y(t)\) on \([0,T]\) is confirmed, which concludes the proof.
\end{proof}

\vspace{-0.5em}

\section{Strategic Inference with Discrete Observations}\label{sec:appendix_c}
In this appendix, we address the observability of \(\sigma_F\), mentioned in Remark~\ref{rem:discrete_time}.
We show that when the follower's optimal state process is only observed at discrete times \(\Delta:0\leq t_1\leq \cdots\leq t_n\leq T\), the joint MLEs for \(M\) and \(\sigma_F\) align with those in our model with continuous observability, i.e., when \(\lambda(\Delta) := \sup_{0\leq i\leq n-1} |t_{i+1} - t_i| \to 0\) as \(n\to\infty\), the joint MLEs converge to their continuous counterparts.

Without knowledge of \(\sigma_F\), the leader jointly estimates \(M\) and \(\sigma_F\) through MLE based on observations of \(x:=(X^{F,*}_{t_1},\ldots,X^{F,*}_{t_n})\).
Since \(X^{F,*}\) is Markov, the joint log-likelihood can be decomposed as \( \ell(x;M,\sigma_F) = \sum_{i=0}^{n-1}\ell(x_{i+1}|x_i;M,\sigma_F)\), where \(x_0\in\R\) denotes the deterministic initial condition \(X^{F,*}_0\). Define \(\mc{F}(a,b):= \int_a^b f(v)\ud v\), where \(f\) is defined in~\eqref{eqn:opt_state}. 
The solution to the Ornstein-Uhlenbeck process~\eqref{eqn:opt_state} has the following representation:
\begin{equation*}
    X^{F,*}_{t_{i+1}} = X^{F,*}_{t_{i}} e^{\mc{F}(t_i,t_{i+1})} - \frac{B_F^2}{R_F}\int_{t_i}^{t_{i+1}} e^{\mc{F}(u,t_{i+1})}b_u\ud u + \sigma_F \int_{t_i}^{t_{i+1}} e^{\mc{F}(u,t_{i+1})}\ud W^F_u,
\end{equation*}
which results in the following log-likelihood: 
\begin{multline*}
    \ell(x;M,\sigma_F) = -\frac{n}{2}\log(2\pi)-n\log\sigma_F-\frac{1}{2}\sum_{i=0}^{n-1}\log \int_{t_i}^{t_{i+1}}e^{2\mc{F}(u,t_{i+1})}\ud u \\
        - \frac{1}{2\sigma_F^2}\sum_{i=0}^{n-1}\frac{[x_{i+1}-x_ie^{\mc{F}(t_i,t_{i+1})}+\frac{B_F^2}{R_F}\int_{t_i}^{t_{i+1}}e^{\mc{F}(u,t_{i+1})}b_u\ud u]^2}{\int_{t_i}^{t_{i+1}}e^{2\mc{F}(u,t_{i+1})}\ud u}.
\end{multline*}
Setting partial derivatives with respect to \(M,\sigma_F\) to zero yields the joint MLEs
\begin{align*}
    \widehat{M} &= -\frac{\sum_{i=0}^{n-1}\frac{[x_{i+1} - x_ie^{\mc{F}(t_i,t_{i+1})}]\int_{t_i}^{t_{i+1}}e^{\mc{F}(u,t_{i+1})}g_u\ud u}{\int_{t_i}^{t_{i+1}}e^{2\mc{F}(u,t_{i+1})}\ud u}}{\frac{B_F^2}{R_F}\sum_{i=0}^{n-1}\frac{(\int_{t_i}^{t_{i+1}}e^{\mc{F}(u,t_{i+1})}g_u\ud u)^2}{\int_{t_i}^{t_{i+1}}e^{2\mc{F}(u,t_{i+1})}\ud u}},\\
    \widehat{\sigma_F^2} &= \frac{1}{n}\sum_{i=0}^{n-1} \frac{[x_{i+1}-x_ie^{\mc{F}(t_i,t_{i+1})}+\widehat{M}\frac{B_F^2}{R_F}\int_{t_i}^{t_{i+1}}e^{\mc{F}(u,t_{i+1})}g_u\ud u]^2}{\int_{t_i}^{t_{i+1}}e^{2\mc{F}(u,t_{i+1})}\ud u},
\end{align*}
where \(g_u := g(u,x^L)\) for brevity.

We use the notation \(a_n\sim b_n\) for two sequences of real numbers, when \(a_n/b_n\to 1\) as \(n\to\infty\) and \(\lambda(\Delta)\to 0\).
Setting \(\lambda(\Delta)\to 0\), on the denominator of \(\widehat M\) we have
\begin{equation*}
    \sum_{i=0}^{n-1}\frac{(\int_{t_i}^{t_{i+1}}e^{\mc{F}(u,t_{i+1})}g_u\ud u)^2}{\int_{t_i}^{t_{i+1}}e^{2\mc{F}(u,t_{i+1})}\ud u}\sim  \sum_{i=0}^{n-1}g^2_{t_i}(t_{i+1}-t_i)\to \int_0^T g_t^2\ud t.
\end{equation*}
On the numerator of \(\widehat{M}\),
\begin{multline*}
    \sum_{i=0}^{n-1}\frac{[X^{F,*}_{t_{i+1}} - X^{F,*}_{t_i}e^{\mc{F}(t_i,t_{i+1})}]\int_{t_i}^{t_{i+1}}e^{\mc{F}(u,t_{i+1})}g_u\ud u}{\int_{t_i}^{t_{i+1}}e^{2\mc{F}(u,t_{i+1})}\ud u}\sim
    \sum_{i=0}^{n-1}g_{t_i}(X^{F,*}_{t_{i+1}} - X^{F,*}_{t_i})\\ - \sum_{i=0}^{n-1}g_{t_i}X^{F,*}_{t_i}(e^{\mc{F}(t_i,t_{i+1})} - 1)\overset{p}{\to} \int_0^T g_t\ud X^{F,*}_t - \int_0^T f(t)g_t X^{F,*}_t\ud t,
\end{multline*}
where \(e^{\mc{F}(t_i,t_{i+1})} - 1 = f(t_i)(t_{i+1} - t_i) + o(\lambda(\Delta))\) follows from the Taylor expansion of \(\eta(x):= e^{\mc{F}(t_i,x)} =  e^{\int_{t_i}^x f(v)\ud v}\), with the properties \(\eta'(x) = \eta(x)f(x)\), \(\eta(t_i) = 1\).
As \(\lambda(\Delta)\to 0\) and \(n\to\infty\), such \(\widehat M\) converges in probability to its counterpart~\eqref{eq:MLE_estimator} under continuous observability.

For the numerator of \(\widehat{\sigma_F^2}\), distribute terms of the square and notice that, due to the presence of \(\frac{1}{n}\), terms that are linear or quadratic in \(t_{i+1}-t_i\) vanish in the limit:
\begin{equation*}
    \widehat{\sigma_F^2}\sim \frac{1}{n}\sum_{i=0}^{n-1} \frac{(X^{F,*}_{t_{i+1}}-X^{F,*}_{t_i}e^{\mc{F}(t_i,t_{i+1})})^2}{t_{i+1} - t_i}\overset{p}{\to} \frac{1}{T}\QV{X^{F,*}}{X^{F,*}}_T \overset{\mathrm{a.s.}}{=} \sigma_F^2.
\end{equation*}
As \(\lambda(\Delta)\to 0\) and \(n\to\infty\), \(\widehat{\sigma_F^2}\) converges in probability to the true underlying \(\sigma_F^2\).

Those observations suggest that, our model with continuous observability can be identified as the limit of the one with discrete observability, which justifies our assumption on the leader's observability of \(\sigma_F\).

\bibliographystyle{siamplain}
\bibliography{references}
\end{document}

%% file: ex_shared.tex
\usepackage{graphicx}
\usepackage{amsmath,amssymb,mathrsfs}
\usepackage{color}
\usepackage{xcolor}
\usepackage{url,hyperref}
\usepackage{graphicx}
\usepackage{comment}
\usepackage{float}
\usepackage{algorithm}
\usepackage{algpseudocode}
\usepackage{booktabs,makecell,diagbox}

\usepackage[color=black!20, textwidth=5.5cm,textsize=footnotesize]{todonotes}

\ifpdf
  \DeclareGraphicsExtensions{.eps,.pdf,.png,.jpg}
\else
  \DeclareGraphicsExtensions{.eps}
\fi

\newsiamthm{example}{Example}
\newsiamremark{remark}{Remark}
\newsiamremark{hypothesis}{Hypothesis}
\crefname{hypothesis}{Hypothesis}{Hypotheses}
\newsiamremark{fact}{Fact}
\crefname{fact}{Fact}{Facts}

\newcommand{\Prob}{\mathbb{P}}
\newcommand{\R}{\mathbb{R}}

\newcommand{\F}{\mathscr{F}}
\newcommand{\E}{\mathbb{E}}
\newcommand{\Var}{\mathrm{Var}}

\newcommand{\ud}{\,\mathrm{d}}
\newcommand{\QV}[2]{\left\langle #1, #2\right\rangle}

\newcommand{\transpose}{^{\operatorname{T}}}
\newcommand{\diag}{\mathrm{diag}}
\newcommand{\mc}[1]{\mathcal{#1}}

\DeclareMathOperator*{\argsup}{arg\,sup}

\headers{Strategic Inference in Stackelberg Games}{R. Hu, D. Ralston, X. Yang and H. Zhou}

\title{Strategic Inference in Stackelberg Games: Optimal Control for Revealing Adversary Intent\thanks{To be submitted.
\funding{This work was partially supported by the ONR grant under \#N00014-24-1-2432. R.H. was also partially supported by a grant from the Simons Foundation (MP-TSM-00002783) and the NSF grant DMS-2420988.}}}

\author{Ruimeng Hu\thanks{Department of Mathematics, and Department of Statistics and Applied Probability, University of California, Santa Barbara, CA 93106 
  (\email{rhu@ucsb.edu}).}
\and Daniel Ralston\thanks{Department of Mathematics, University of California, Santa Barbara, CA 93106
  (\email{danielralston@ucsb.edu}, \email{xuyang@math.ucsb.edu}).}
\and Xu Yang\footnotemark[3]
\and Haosheng Zhou\thanks{Department of Statistics and Applied Probability, University of California, Santa Barbara, CA 93106. Current: Department of Operations Research and Financial Engineering, Princeton University, Princeton, NJ 08544 
(\email{hz9769@princeton.edu}).}
}

%% file: article_final.bbl
\begin{thebibliography}{10}

\bibitem{abou2012matrix}
{\sc H.~Abou-Kandil, G.~Freiling, V.~Ionescu, and G.~Jank}, {\em Matrix {R}iccati {E}quations in {C}ontrol and {S}ystems {T}heory}, Birkh{\"a}user, Basel, 2012.

\bibitem{adams_survey_2022}
{\sc S.~Adams, T.~Cody, and P.~A. Beling}, {\em A survey of inverse reinforcement learning}, Artif. Intell. Rev., 55 (2022), pp.~4307--4346.

\bibitem{arora_survey_2020}
{\sc S.~Arora and P.~Doshi}, {\em A survey of inverse reinforcement learning: {C}hallenges, methods and progress}, Artificial Intelligence, 297 (2021), p.~103500.

\bibitem{bagchi_stackelberg_1981}
{\sc A.~Bagchi and T.~Ba{\c{s}}ar}, {\em Stackelberg strategies in linear-quadratic stochastic differential games}, J. Optim. Theory Appl., 35 (1981), pp.~443--464.

\bibitem{basar_dynamic_1998}
{\sc T.~Ba{\c{s}}ar and G.~J. Olsder}, {\em Dynamic {N}oncooperative {G}ame {T}heory}, SIAM, Philadelphia, 1999.

\bibitem{bensoussan2015maximum}
{\sc A.~Bensoussan, S.~Chen, and S.~P. Sethi}, {\em The {M}aximum {P}rinciple for {G}lobal {S}olutions of {S}tochastic {S}tackelberg {D}ifferential {G}ames}, SIAM J. Control Optim., 53 (2015), pp.~1956--1981.

\bibitem{bergemann_information_2019}
{\sc D.~Bergemann and S.~Morris}, {\em Information {D}esign: {A} {U}nified {P}erspective}, J. Econ. Lit., 57 (2019), pp.~44--95.

\bibitem{bock_parameter_2013}
{\sc H.~G. Bock, S.~K{\"o}rkel, and J.~P. Schl{\"o}der}, {\em Parameter {E}stimation and {O}ptimum {E}xperimental {D}esign for {D}ifferential {E}quation {M}odels}, in Model {B}ased {P}arameter {E}stimation: {T}heory and {A}pplications, H.~G. Bock, T.~Carraro, W.~J{\"a}ger, S.~K{\"o}rkel, R.~Rannacher, and J.~P. Schl{\"o}der, eds., Springer, Berlin, 2012, pp.~1--30.

\bibitem{caines2021graphon}
{\sc P.~E. Caines and M.~Huang}, {\em Graphon {M}ean {F}ield {G}ames and {T}heir {E}quations}, SIAM J. Control Optim., 59 (2021), pp.~4373--4399.

\bibitem{carmona2025reconciling}
{\sc R.~Carmona and M.~Lauriere}, {\em Reconciling {D}iscrete-{T}ime {M}ixed {P}olicies and {C}ontinuous-{T}ime {R}elaxed {C}ontrols in {R}einforcement {L}earning and {S}tochastic {C}ontrol}.
\newblock preprint, arXiv:2504.21793 [math.OC], 2025.

\bibitem{chen2019adversarial}
{\sc T.~Chen, J.~Liu, Y.~Xiang, W.~Niu, E.~Tong, and Z.~Han}, {\em Adversarial attack and defense in reinforcement learning-from {AI} security view}, Cybersecur., 2 (2019), pp.~1--22.

\bibitem{jb_cruz_survey_1975}
{\sc J.~Cruz, Jr.}, {\em Survey of {N}ash and {S}tackelberg {E}quilibrim {S}trategies in {D}ynamic {G}ames}, in Annals of Economic and Social Measurement, Volume 4, number 2, S.~V. Berg, ed., National Bureau of Economic Research, Cambridge, MA, 1975, pp.~339--344.

\bibitem{eysenbach2021maximum}
{\sc B.~Eysenbach and S.~Levine}, {\em Maximum {E}ntropy {RL} ({P}rovably) {S}olves {S}ome {R}obust {RL} {P}roblems}.
\newblock preprint, arXiv:2103.06257 [cs.LG], 2021.

\bibitem{feldbaum1960dual}
{\sc A.~A. Feldbaum}, {\em Dual {C}ontrol {T}heory. {I}}, Avtomat. i Telemekh., 21 (1960), pp.~1240--1249.

\bibitem{fleming1984stochastic}
{\sc W.~H. Fleming and M.~Nisio}, {\em On stochastic relaxed control for partially observed diffusions}, Nagoya Math. J., 93 (1984), pp.~71--108.

\bibitem{florens1989approximate}
{\sc D.~Florens-Zmirou}, {\em Approximate {D}iscrete-{T}ime {S}chemes for {S}tatistics of {D}iffusion {P}rocesses}, Statistics, 20 (1989), pp.~547--557.

\bibitem{frikha2023actor}
{\sc N.~Frikha, M.~Germain, M.~Lauri{\`e}re, H.~Pham, and X.~Song}, {\em Actor-{C}ritic learning for mean-field control in continuous time}.
\newblock preprint, arXiv:2303.06993 [stat.ML], 2023.

\bibitem{hale2009ordinary}
{\sc J.~K. Hale}, {\em Ordinary {D}ifferential {E}quations}, Dover, Mineola, NY, 2009.

\bibitem{han2016deep}
{\sc J.~Han and W.~E}, {\em Deep {L}earning {A}pproximation for {S}tochastic {C}ontrol {P}roblems}.
\newblock preprint, arXiv:1611.07422 [cs.LG], 2016.

\bibitem{HanHu2021}
{\sc J.~Han and R.~Hu}, {\em Recurrent neural networks for stochastic control problems with delay}, Math. Control Signals Systems, 33 (2021), pp.~775--795.

\bibitem{hazan2019provably}
{\sc E.~Hazan, S.~Kakade, K.~Singh, and A.~Van~Soest}, {\em Provably {E}fficient {M}aximum {E}ntropy {E}xploration}, in Proceedings of the 36th International Conference on Machine Learning, PMLR, 2019, pp.~2681--2691.

\bibitem{hernandez_closed-loop_2024}
{\sc C.~Hern{\'a}ndez, N.~H. Santib{\'a}nez, E.~Hubert, and D.~Possama{\"\i}}, {\em Closed-loop equilibria for {S}tackelberg games: it's all about stochastic targets}.
\newblock preprint, arXiv:2406.19607 [math.OC], 2024.

\bibitem{hochreiter1997long}
{\sc S.~Hochreiter and J.~Schmidhuber}, {\em Long {S}hort-{T}erm {M}emory}, Neural Comput., 9 (1997), pp.~1735--1780.

\bibitem{hooker2015control}
{\sc G.~Hooker, K.~K. Lin, and B.~Rogers}, {\em Control {T}heory and {E}xperimental {D}esign in {D}iffusion {P}rocesses}, SIAM/ASA J. Uncertain. Quantif., 3 (2015), pp.~234--264.

\bibitem{hu2025finite}
{\sc R.~Hu, J.~Long, and H.~Zhou}, {\em Finite-{A}gent {S}tochastic {D}ifferential {G}ames on {L}arge {G}raphs: {I}. {T}he {L}inear-{Q}uadratic {C}ase}, Appl. Math. Optim., 92 (2025), pp.~1--50.

\bibitem{jaynes1957information}
{\sc E.~T. Jaynes}, {\em Information {T}heory and {S}tatistical {M}echanics}, Phys. Rev., 106 (1957), pp.~620--630.

\bibitem{kim2025deceptive}
{\sc Y.~Kim, A.~Benvenuti, B.~Chen, M.~Karabag, A.~Kulkarni, N.~D. Bastian, U.~Topcu, and M.~Hale}, {\em Deceptive {S}equential {D}ecision-{M}aking via {R}egularized {P}olicy {O}ptimization}.
\newblock preprint, arXiv:2501.18803 [math.OC], 2025.

\bibitem{kutoyants2013statistical}
{\sc Y.~A. Kutoyants}, {\em Statistical {I}nference for {E}rgodic {D}iffusion {P}rocesses}, Springer, London, 2013.

\bibitem{li2023linear}
{\sc Z.~Li and J.~Shi}, {\em Linear {Q}uadratic {L}eader-{F}ollower {S}tochastic {D}ifferential {G}ames: {C}losed-{L}oop {S}olvability}, J. Syst. Sci. Complex., 36 (2023), pp.~1373--1406.

\bibitem{li2024closed}
{\sc Z.~Li and J.~Shi}, {\em Closed-{L}oop {S}olvability of {L}inear {Q}uadratic {M}ean-{F}ield {T}ype {S}tackelberg {S}tochastic {D}ifferential {G}ames}, Appl. Math. Optim., 90 (2024), p.~22.

\bibitem{liuIncentivebasedmodeling2005}
{\sc P.~Liu, W.~Zang, and M.~Yu}, {\em Incentive-based modeling and inference of attacker intent, objectives, and strategies}, ACM Trans. Inf. Syst. Secur., 8 (2005), pp.~78--118.

\bibitem{lopes_active_2009}
{\sc M.~Lopes, F.~Melo, and L.~Montesano}, {\em Active {L}earning for {R}eward {E}stimation in {I}nverse {R}einforcement {L}earning}, in Joint European Conference on Machine Learning and Knowledge Discovery in Databases, Springer, 2009, pp.~31--46.

\bibitem{mesbah2018stochastic}
{\sc A.~Mesbah}, {\em Stochastic model predictive control with active uncertainty learning: {A} {S}urvey on dual control}, Annu. Rev. Control, 45 (2018), pp.~107--117.

\bibitem{moon_linear-quadratic_2021}
{\sc J.~Moon}, {\em Linear-{Q}uadratic {S}tochastic {S}tackelberg {D}ifferential {G}ames for {J}ump-{D}iffusion {S}ystems}, SIAM J. Control Optim., 59 (2021), pp.~954--976.

\bibitem{ng_algorithms_2000}
{\sc A.~Y. Ng, S.~Russell, et~al.}, {\em Algorithms for {I}nverse {R}einforcement {L}earning}, in Proceedings of the Seventeenth International Conference on Machine Learning, ICML, 2000, pp.~663--670.

\bibitem{nicole1987compactification}
{\sc K.~Nicole~el, N.~Du'h{\=U}{\=U}, and J.-P. Monique}, {\em Compactification methods in the control of degenerate diffusions: existence of an optimal control}, Stochastics, 20 (1987), pp.~169--219.

\bibitem{oksendal_stochastic_2013}
{\sc B.~{\O}ksendal, L.~Sandal, and J.~Ub{\o}e}, {\em Stochastic {S}tackelberg equilibria with applications to time-dependent newsvendor models}, J. Econom. Dynam. Control, 37 (2013), pp.~1284--1299.

\bibitem{papavassilopoulos2003existence}
{\sc G.~Papavassilopoulos and J.~Cruz}, {\em On the existence of solutions to coupled matrix {R}iccati differential equations in linear quadratic {N}ash games}, IEEE Trans. Automat. Control, 24 (1979), pp.~127--129.

\bibitem{pukelsheim_optimal_2006}
{\sc F.~Pukelsheim}, {\em Optimal {D}esign of {E}xperiments}, SIAM, Philadelphia, 2006.

\bibitem{rainforth_modern_2023}
{\sc T.~Rainforth, A.~Foster, D.~R. Ivanova, and F.~Bickford~Smith}, {\em Modern {B}ayesian {E}xperimental {D}esign}, Statist. Sci., 39 (2024), pp.~100--114.

\bibitem{shinde2021cyber}
{\sc A.~Shinde, P.~Doshi, and O.~Setayeshfar}, {\em Cyber {A}ttack {I}ntent {R}ecognition and {A}ctive {D}eception using {F}actored {I}nteractive {POMDP}s}, in Proceedings of the 20th International Conference on Autonomous Agents and MultiAgent Systems, AAMAS, 2021, pp.~1200--1208.

\bibitem{silvey_optimal_2013}
{\sc S.~Silvey}, {\em Optimal {D}esign: {A}n {I}ntroduction to the {T}heory for {P}arameter {E}stimation}, Springer, Netherlands, 2013.

\bibitem{simaan_stackelberg_1973}
{\sc M.~Simaan and J.~Cruz, Jr.}, {\em On the {S}tackelberg strategy in nonzero-sum games}, J. Optim. Theory Appl., 11 (1973), pp.~533--555.

\bibitem{sun2006exact}
{\sc Y.~Sun}, {\em The exact law of large numbers via {F}ubini extension and characterization of insurable risks}, J. Econom. Theory, 126 (2006), pp.~31--69.

\bibitem{sun2009individual}
{\sc Y.~Sun and Y.~Zhang}, {\em Individual risk and {L}ebesgue extension without aggregate uncertainty}, J. Econom. Theory, 144 (2009), pp.~432--443.

\bibitem{wang2020reinforcement}
{\sc H.~Wang, T.~Zariphopoulou, and X.~Y. Zhou}, {\em Reinforcement {L}earning in {C}ontinuous {T}ime and {S}pace: {A} {S}tochastic {C}ontrol {A}pproach}, J. Mach. Learn. Res., 21 (2020), pp.~1--34.

\bibitem{ward2023active}
{\sc W.~Ward, Y.~Yu, J.~Levy, N.~Mehr, D.~Fridovich-Keil, and U.~Topcu}, {\em Active {I}nverse {L}earning in {S}tackelberg {T}rajectory {G}ames}, in American Control Conference, IEEE, 2025, pp.~1547--1553.

\bibitem{wilson_trajectory_2014}
{\sc A.~D. Wilson, J.~A. Schultz, and T.~D. Murphey}, {\em Trajectory {S}ynthesis for {F}isher {I}nformation {M}aximization}, IEEE Trans. Robot., 30 (2014), pp.~1358--1370.

\bibitem{wittenmark1995adaptive}
{\sc B.~Wittenmark}, {\em Adaptive dual control methods: {A}n overview}, in Adaptive Systems in Control and Signal Processing, Elsevier, 1995, pp.~67--72.

\bibitem{yong_leader-follower_2002}
{\sc J.~Yong}, {\em A {L}eader-{F}ollower {S}tochastic {L}inear {Q}uadratic {D}ifferential {G}ame}, SIAM J. Control Optim., 41 (2002), pp.~1015--1041.

\bibitem{zhou2025adversarial}
{\sc H.~Zhou, D.~Ralston, X.~Yang, and R.~Hu}, {\em {A}dversarial {D}ecision-{M}aking in {P}artially {O}bservable {M}ulti-{A}gent {S}ystems: {A} {S}equential {H}ypothesis {T}esting {A}pproach}.
\newblock preprint, arXiv:2509.03727 [math.OC], 2025.

\bibitem{zhou2025integrating}
{\sc H.~Zhou, D.~Ralston, X.~Yang, and R.~Hu}, {\em Integrating {S}equential {H}ypothesis {T}esting into {A}dversarial {G}ames: {A} {S}un {Z}i-{I}nspired {F}ramework}.
\newblock preprint, arXiv:2502.13462 [math.OC], 2025.

\bibitem{zhou1992existence}
{\sc X.~Y. Zhou}, {\em On the {E}xistence of {O}ptimal {R}elaxed {C}ontrols of {S}tochastic {P}artial {D}ifferential {E}quations}, SIAM J. Control Optim., 30 (1992), pp.~247--261.

\end{thebibliography}
